\documentclass[twoside]{article}

\usepackage{microtype}
\usepackage{graphicx}
\usepackage{subfigure}
\usepackage{booktabs} 

\usepackage{hyperref}
\usepackage{float}
\usepackage{amsmath}
\usepackage{amssymb}
\usepackage{mathtools}
\usepackage{amsfonts}
\usepackage{enumitem}
\usepackage{amsthm}

\usepackage[textwidth=2.9cm, textsize=tiny]{todonotes}

\usepackage{longtable}
\usepackage{supertabular}

\usepackage[ruled,vlined]{algorithm2e}
\usepackage{tikz}
\usetikzlibrary{decorations.pathreplacing}
 
\usepackage{xcolor}         
\definecolor{linkcolor}{RGB}{83,83,182}
\definecolor{citecolor}{RGB}{128,0,128}
\hypersetup{
    colorlinks=true,
    citecolor=citecolor,
    linkcolor=linkcolor,
    urlcolor=linkcolor
}

\newcommand{\abs}[1]{\lvert #1 \rvert}

\newcommand{\ie}{{\em i.e.,~}}

\usepackage[capitalize,noabbrev]{cleveref}

\newtheorem{theorem}{Theorem}[section]
\newtheorem{proposition}[theorem]{Proposition}

\newtheorem{corollary}[theorem]{Corollary}

\newtheorem{definition}[theorem]{Definition}
\newtheorem{assumption}[theorem]{Assumption}

\usepackage[accepted]{aistats2024}
\usepackage[round]{natbib}

\begin{document}

%

%

\twocolumn[

\aistatstitle{Provable local learning rule by expert aggregation for a Hawkes network}

\aistatsauthor{ Sophie Jaffard$^1$ \And Samuel Vaiter$^1$ \And  Alexandre Muzy$^2$ \And Patricia Reynaud-Bouret$^1$ }

\aistatsaddress{ $^1$Université Côte d'Azur, CNRS, LJAD \And $^2$Université Côte d'Azur, CNRS, i3S } ]

\begin{abstract}
  We propose a simple network of Hawkes processes as a cognitive model capable of learning to classify objects. Our learning algorithm, named HAN for Hawkes Aggregation of Neurons, is based on a local synaptic learning rule based on spiking probabilities at each output node. We were able to use local regret bounds to prove mathematically that the network is able to learn on average and even asymptotically under more restrictive assumptions.
\end{abstract}

\section{INTRODUCTION}
Recordings of human brain suggest that concepts are represented through sparse set of neurons that fire when the concept is activated \citep{legenstein2016variable}. In this sense, Spiking Neural Networks \citep{TAVANAEI201947} are more biologically plausible than Artificial Neural Networks if one wants to understand how the brain encodes information at neuronal level, and stochastic modelling is particularly relevant \citep{Buesing2011}. 

Neuroscientists have identified local learning rules to adjust synaptic weights, regrouped in the concept of  Spike-Timing-Dependent Plasticity (STDP) process~\citep{TAVANAEI201947,caporale, NIPS2009_a5cdd4aa}. This is a form of Hebbian learning \citep{hebb}, where connections between neurons are strengthened or weakened depending on their relative spike times in a short time-window. Other biological rules have been used, for instance to model the olfactory system \citep{kepple2019deconstructing}. However, to our knowledge there is no mathematical proof that such local rules enable to learn. If there was, it would help in understanding how local transformations can lead to global learning. 

Hawkes processes \citep{Hawkes1971} are point processes that are frequently used as models in a variety of settings: network analysis, financial transactions, seismic or health data \citep{hall2016tracking,zuo2020transformer}. In particular, a classic application consists in modeling  interactions between neurons \citep{reynaud2013inference,lambert2018reconstructing}. Many works deal with estimation in these models \citep{yang2017online,wang2020uncertainty}, sometimes using  recurrent neural networks \citep{sharma2019generative,zhang2020self}. Simulation of
large networks of these processes is also widely studied in the literature  \citep{bacry2017tick,phi2020event,mascartetal2022,10.1145/3565809}.
Generalizations of these interaction models
 have also been studied for estimation purposes
using deep networks \citep{mei2017neural,zuo2020transformer}.

Our purpose in the present work is totally different from estimation or simulation. As a first step towards proving mathematically that bio-inspired networks using local learning rules can learn, we use Hawkes networks as a model for a cognitive network that can provably learn to classify objects into one of several categories by updating synaptic weights with a local learning rule. See an illustrative example in Figure \ref{net}. In this network, the output nodes are post-synaptic neurons that  produce spikes as a discrete-time Hawkes process \citep{ost2020sparse, bremaud1996}, whose spiking probability is a function of the weighted sum of the activity of the pre-synaptic neurons at the previous time step. In the case of a linear Hawkes process, Kalikow decomposition \citep{ost2020sparse} allows us to interpret these synaptic  weights in the previous sum as a probability distribution. In particular, it is possible to  randomly choose  the presynaptic neuron of interest instead of doing the whole sum over all presynaptic neurons.

\begin{figure} \label{net}
   \hspace{-0.5cm}
    \centering
    \includegraphics[width=5cm]{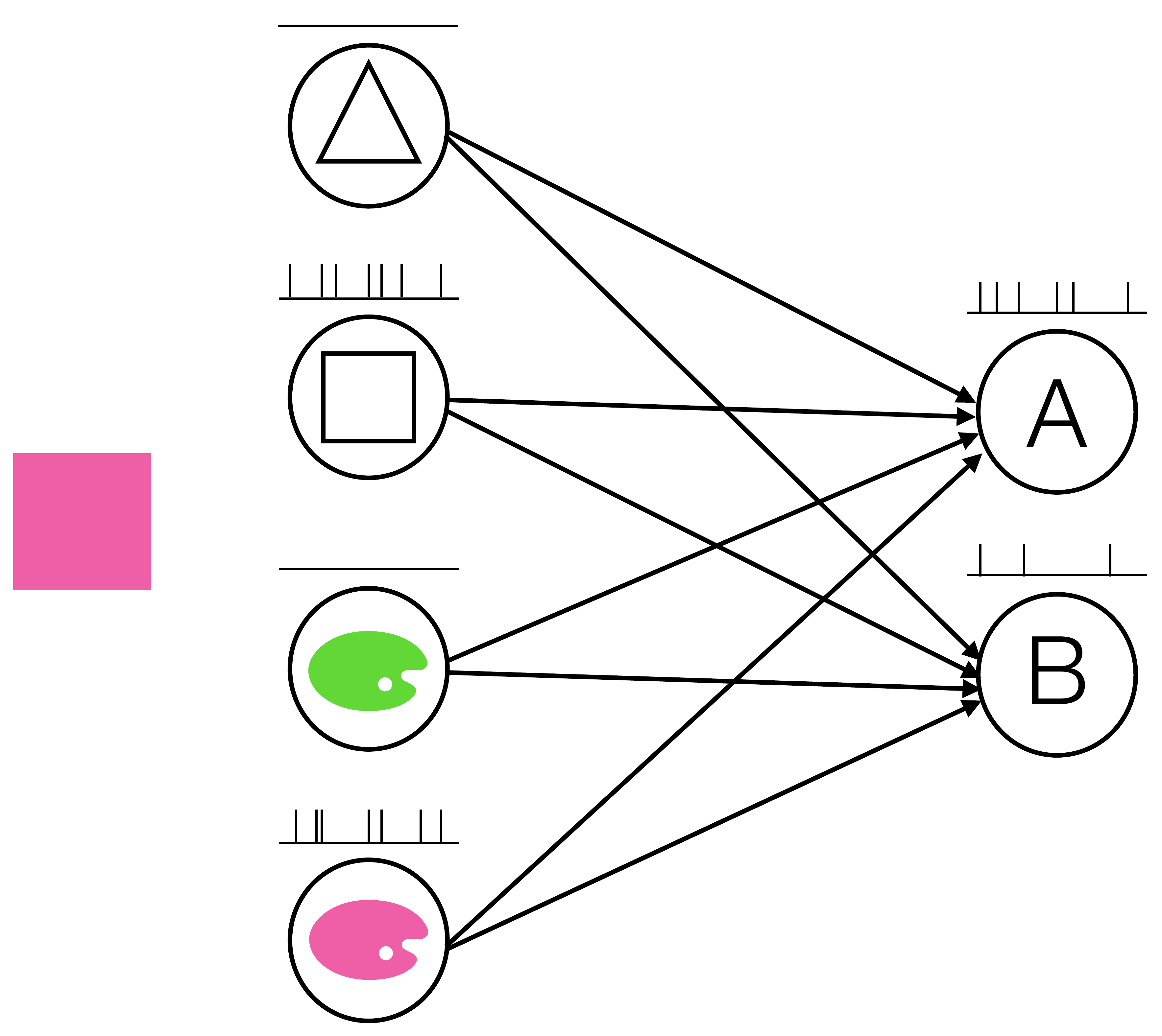}
    \caption{Illustrative example of the network. The presented object excites the neurons encoding its features. Then it is classified in the class coded by the output neuron which spiked the most, here class $A$. }
    \label{curves}
\end{figure}

This interpretation of the synaptic weights leads to the following local vision: for an output neuron, its presynaptic neurons can be seen as so many experts and the distribution, given by the weights, can be related to an expert aggregation problem. This is why we use at this stage an expert aggregation algorithm \citep{cesa2006prediction}, or also known as mixture of experts (MoE), to update the weights. However, the key ingredient is that gains of presynaptic neurons are not arbitrary, as usual in expert aggregation, but are selected depending on their spiking probability. 

The resulting algorithm is called HAN (Hawkes Aggregation of Neurons), and is general enough for any expert aggregation algorithm.
\newline

\textbf{Contributions.} We propose a Hawkes network that learns to classify objects with a local learning rule. More precisely, our contributions are the following:\\
$\bullet$ We interpret the optimization of synaptic weights as small \textbf{expert aggregation} problems that are solved locally by each output neuron, see Algorithm \ref{waca} (HAN).\\
$\bullet$ In the case of a linear Hawkes process, we prove that the network learns to correctly classify objects on average for any expert aggregation algorithm verifying a certain regret bound (Theorem \ref{theorem oracle}). More precisely, we have an \textbf{oracle inequality} (with constant 1): the obtained network has the same {\bf network discrepancy} in spiking probability as the best possible network up to an additive error in $O(M^{-1/2})$, where $M$ is the number of objects presented to the network during its learning phase.\\
$\bullet$ In the case of a general Hawkes process, we explicitly compute the limit of the weights when using the specific expert aggregation algorithm EWA (Exponentially Weighted Average) \citep{cesa1999prediction,cesa2006prediction,stoltz2010agregation}. 
\newline

\textbf{Related Work. } The proposed network is inspired by the Component-Cue model \citep{gluck1988conditioning}. In this cognitive model, the objects classified by the network have several features, and the network learns to classify them in the right category by learning combinations of features which  predict correctly the category. The features of the objects to classify are represented by input nodes and their categories by output nodes. \cite{MEZZADRI2022102691, mezzadri:tel-03219311} compared this model to the ALCOVE model \citep{kruschke2020alcove}, where objects are classified according to their similarities with previously learned objects, and showed that the Component-Cue model is most of the time a better fit to human learning than the ALCOVE model. At the difference with the present work, the original Component-Cue model does not incorporate firing patterns of neurons nor local learning rule.

Kalikow decomposition has been mainly used  to prove existence of stationary processes \citep{galves2013infinite,galves2013kalikow,hodara2017hawkes}. Recently it has been used \citep{phi2020event} to simulate neurons in interaction with a potentially infinite  neural network.

Online learning in a context of a Hawkes network has been used by \cite{hall2016tracking}, where a dynamic mirror descent is performed to track how events influence future events, and by \cite{yang2017online}, to estimate the triggering functions of the processes. However, in these works, online learning has been used to estimate the parameters of the Hawkes processes, whereas in the present work, online learning is used to update synaptic weights to make the Hawkes network learn how to classify objects by itself.

In neuroscience, two main local synaptic rules have been proposed to link a behavior to a corresponding synaptic mechanism and STDP is not one of them. 
The three-factors rule~\citep{three-factor} assumes that a synaptic weight update depends on (i) the presynaptic activation, (ii) the post-synaptic activation, and (iii) the eventual outcome of the overall behavior. The rate-based learning rule~\citep{rate-based-rule} assumes that the weight update depends on the firing rate of pre- and post-synaptic neurons. Our local rule is closer to this approach than to STDP. To the best of our knowledge, the present work mathematically proves for the first time that such local learning rules make a very simple network learn.


\section{FRAMEWORK}

\subsection{First notations and set-up}

All the notations are listed in Appendix \ref{app notations}. The objects to be classified have different natures $o\in \mathcal{O}$, each $o$ having different features, each feature being the version of a general characteristic. For instance, a blue square can have the feature "blue" which characteristic is color.

We work in discrete time. A number $M$ of objects are successively presented to the network, each for $N$ time steps. We denote by $o(m)\in \mathcal{O}$ the nature of the object presented to the network during the $m^{th}$ round. We want the network to learn to classify the objects in classes. For class $j$, $M^j$ is the number of objects  belonging to class $j$ in the first $M$ presented objects. 

The present network is made of two layers. We denote by $I$ the set of input neurons, and $J$ the set of 
output neurons. Each output neuron $j$ corresponds to a class, also noted $j$, in which one wants 
to classify the objects that are shown. Then the network activity is a sequence of random variables 
$(X^i_{m,t})_{i\in I\cup J, 1\leq m \leq M, 1\leq t \leq N}$ where
\[X^i_{m,t} = \left\{
    \begin{array}{ll}
    1 \text{ if neuron $i$ spiked at time $t$ for object $m$} \\
    0 \text{ otherwise.}
    \end{array}
\right. \]

\subsection{The input layer} \label{sec input}
Each neuron of the input layer emits a discrete process that is a sequence of independent random variables following a Bernoulli distribution (note that we assume temporal independence for one neuron but not independence between neurons). The parameter of the Bernoulli distribution is denoted
\[p^i_m:=\mathbb{P}(X^i_{m,1}=1),\]
hence $p^i_m$ is the spiking probability of neuron $i$ at any time step when presented with the $m^{th}$ object (see Step 4 of Algorithm \ref{waca}). We assume that $p^i_m$ does not depend on $m$ per see but only on eventually the nature $o(m)\in \mathcal{O}$. For instance, if $p^i_m$ corresponds to an input neuron $i$ detecting the feature "blue" then $p_m^i$ will be the same for all objects with the same feature blue. It could potentially be different for each $o$, but we are interested in practise by the case where input neurons describe objects through their features, not their nature, as in the Component-Cue model, which describes well human learning \citep{MEZZADRI2022102691}.

\subsection{The output layer} \label{sec output layer}

The output layer is made of neurons coding for the classes in which the objects are classified, so $J$ is used for both the set of output neurons and classes. 

The {\it rule for classifying the objects} is as follows: the object is classified in the class coded by the output neuron which spiked the most during its presentation, \ie in $\arg\max_{j\in J} \widehat{p^j_m}$ where $\widehat{p^j_m}:= \sum_{t=1}^N X_{m,t}^j/N$.

Each input neuron can impact output neuron $j$ via an inhibitory or excitatory connection. To make the distinction, we consider now the set of signed input neurons $I^{+/-}:= I\times \{+,-\}$. From now we denote $i^+=(i,+)\in I^{+/-}$ an excitatory connection and $i^-=(i,-)\in I^{+/-}$ an inhibitory connection. It is possible that a given $I$ has two connections with $j$ via its $i^+$ and $i^-$. These two presynaptic connections will be in this case considered as two experts for $j$. Let $I^j_+\subset I^{+/-}$ be all the excitatory presynaptic neurons of $j$ and $I^j_-\subset I^{+/-}$ be all the inhibitory presynaptic neurons of $j$ such that $I^j:=I^j_+\cup I^j_-$ is the set of connections and experts of $j$.

 We assume that input neurons start spiking $K$ time steps before output neurons. Then the conditional spiking probability of neuron $j$ at time $t$ of object $m$ knowing the network past activity is given by 
\begin{align}
    p^{j,\text{cond}}_{m,t}(w^j_m) := &\varphi\Big(\alpha^j + \sum_{i \in I^{j}_+} w^{i^+\to j}_m \sum_{k=1}^K g_+(k) X^i_{m,t-k} \nonumber \\
    &-\sum_{i \in I^{j}_-} w^{i^-\to j}_m \sum_{k=1}^K g_-(k) X^i_{m,t-k} \Big)  \label{Hawkes} 
\end{align}
where $\varphi : \mathbb{R} \mapsto \mathbb{R}$ is a Lipschitz function, $\alpha^j$ is the spontaneous activity of neuron $j$, $w^{i^+\to j}_m$ (resp. $w^{i^-\to j}_m$) is the weight of the excitatory (resp. inhibitory) connection from neuron $i$ to $j$, $w^j_m:=(w^{i^\bullet \to j}_m)_{i^\bullet \in I^j}$ is the weight family, and $g_+$ and $g_-$ are functions representing the dependency on the past. We denote by $w_{1:M}^j:=(w^j_m)_{1\leq m \leq M}$ the total family of synaptic weights of neuron $j$.

Synaptic weights are updated after every time period during which an object $o(m)$ is presented, so they depend  on $m$. Moreover they represent a probability distribution, that is: for all $j\in J$, $i^\bullet \in I^j$, $1\leq m \leq M$,
$w^{i^\bullet \to j}_m > 0$, and $\sum_{l\in I^j}w^{i^\bullet \to j}_m = 1$.
Besides, for all $\bullet\in \{+,-\}$, $g_\bullet$ is such that $g_\bullet(k)\geq 0$ and $\sum_{k=1}^K g_\bullet(k) = 1$. The parameters $\varphi$ and $\alpha^j$ are chosen such that $p^{j,\text{cond}}_{m,t}(w^j_m)\in [0,1]$ a.s. for any weights $w^j_m$.

In section \ref{average}, we consider the linear case $\varphi=\text{Id}$, $\alpha^j=0$ and $I^j_-=\emptyset$ 
(the last two criterion ensuring that $p^{j,\text{cond}}_{m,t}(w^j_m)\in [0,1]$ a.s.). In this framework,
 one can simulate $X^j_{m,t}$ thanks to the Solo steps of Algorithm \ref{waca}: 
 the simulation of the activity of only one input neuron is needed. Hence in this case the algorithm
 is called HAN Solo. The fact that this method indeed gives a process satisfying \eqref{Hawkes} comes from 
 the Kalikow decomposition of the Hawkes process. For more details about the legitimacy of this 
 operation and the particular case of discrete Hawkes processes, we refer the reader to Section 4.2.3
  of \cite{ost2020sparse} and Appendix \ref{sec Kalikow}. 

\subsection{Learning rule} \label{learning rule}

We use an expert aggregation algorithm to update the weights. One interpretation of the expert aggregation problem \citep{cesa2006prediction} is as follows: a forecaster can choose between several experts, each with an unknown gain, during $M$ rounds. In each round, the forecaster defines a strategy, \ie a probability distribution over the set of experts, and receives the corresponding aggregate sum of the gains. An expert aggregation algorithm is a function used to update the probability distribution in order to maximize gains. This update depends only on past gains. 

Here, each output neuron $j$ is a forecaster, and the experts are its connections to input neurons $I^j$. A round corresponds to the presentation of an object, and the synaptic weights, that we reinterpret as a probability distribution thanks to the intuition given by Kalikow decomposition in the linear case, correspond to the probability distribution chosen by neuron $j$ in the expert aggregation. The gain of connection $i^\bullet$ (the expert) w.r.t. the output neuron $j$ (the forecaster) at round $m$ is denoted by $g^{i^\bullet \to j}_m$ and is defined precisely in the next section. We denote by $G^j_m:=\sum_{m'=1}^m \sum_{i^\bullet\in I_j} w_{m'}^{i^\bullet \to j} g_{m'}^{i^\bullet \to j}$ the cumulated gain of output neuron $j$ until round $m$, and $G_m^{i^\bullet \to j}:=\sum_{m'=1}^m g_{m'}^{i^\bullet \to j}$ the cumulated gain of connection $i^\bullet$. 

The weights of neuron $j$ for the next round are then updated thanks to previously acquired knowledge about the gains, that is, $G^j_m$ and $(G^{i^\bullet \to j}_m)_{i^\bullet\in I_j}$:
\begin{equation} \label{func agg}
    w^j_{m+1}= f(G^j_m,(G^{i^\bullet \to j}_m)_{i^\bullet\in I_j})
\end{equation}
where the function $f$ is the expert aggregation algorithm. Let us give two examples of such algorithms (see details and other algorithms in \cite{cesa2006prediction}) that we will use in sections \ref{limit} and \ref{concrete ex}: \\
$\bullet$ EWA (Exponentially Weighted Average)
\begin{equation}\label{EWA}
    w^{i^\bullet \to j}_{m+1} = \frac{\exp\Big(\eta^j  G^{i^\bullet \to j}_m \Big)}{\sum_{k\in I^j} \exp\Big(\eta^j G^{k\to j}_m \Big)}.
\end{equation} 
The parameter $\eta^j$ is called the learning rate.\\
$\bullet$ PWA (Polynomially Weighted Average) 
\begin{equation} \label{PWA}
    w_{m+1}^{i^\bullet \to j} = \frac{(G_m^{i^\bullet \to j}-G_m^j)_+^{\beta^j-1}}{\sum_{k\in I_j} (G_m^{k\to j}- G_m^j)_+^{\beta^j-1}}
\end{equation}
where $\beta^j\geq 2$ is a parameter to choose. \\
Note that EWA and PWA implement two different strategies: EWA only takes into account the cumulated gains of the experts (\ie connections $i^+$ and $i^-$) and assigns strictly positive weights, whereas PWA compares them with those of the forecaster (\ie neuron $j$), and as soon as an expert is outclassed by the forecaster, it is assigned a weight equal to zero. 

This defines Algorithm \ref{waca} (HAN, for Hawkes Aggregation of neurons), which is given for any expert aggregation algorithm $f$. The algorithm's complexity is determined by the number of calls made to the pseudorandom generator for obtaining Bernoulli variables (and potentially the cost of expert update). With $D$ the output nodes degree, we need to simulate 
$\mathcal{O}(NM(|J|+|I|)$ Bernoulli r.v., and perform 
$\mathcal{O}(NM|J|KD)$ (resp. $\mathcal{O}(M(|J|D + N|I|))$) elementary operations (scalar addition, exponential) for HAN (resp. HAN-Solo).

\begin{algorithm}
    \textbf{Initialization:} $G^{i^\bullet \to j}_0:=0$, $w^{i^\bullet \to j}_1:=1/|I^j|$ \\ 
    \nl
    \For{$m=1$ {\bfseries to} $M$}{ \nl
        \For{$t=1$ {\bfseries to} $K$}{ \nl
            \For{$i\in I$}{ \nl
                $X^i_{m,t} \sim \mathcal{B}(p^i_m)$. 
            } 
            }\nl
         \For{$t=K+1$ {\bfseries to} $N$}{ \nl
            \For{$i\in I$}{ \nl
                $X^i_{m,t} \sim \mathcal{B}(p^i_m)$. 
            } \nl
            \For{$j\in J$}{ \nl
                     $X^j_{m,t} \sim \mathcal{B}(p^{j,\text{cond}}_{m,t})$   
            }
        } \nl
        \For{$j\in J$}{ \nl
            \For{$i^\bullet\in I^j$}{ \nl
            Compute $g^{i^\bullet \to j}_m$ according to \eqref{formule credit}. 
            } \nl
            $w^j_{m+1} \xleftarrow[]{}  f(G^j_m,(G^{i^\bullet \to j}_m)_{i^\bullet\in I^j})$ \tcp{aggregate experts using \eqref{func agg}}
        }
    } 
    \vspace{-2mm}
    \textbf{Output}: $(\arg\max_j \widehat{p^j_m})_{m}$ \tcp{classifications}

    \vspace{-4.8cm}
    \hspace{4.45cm} \textbf{Solo steps}  \\
    \hspace{4.6cm} $\hat{i}^+ \sim w_m^j$ \\ 
    \hspace{4.6cm} $\hat{k} \sim g_+$ \\ 
    \hspace{4.6cm} $X^j_{m,t} \xleftarrow[]{} X^{\hat{i}}_{m,t-\hat{k}}$ \\

    \begin{tikzpicture}[overlay, remember picture]
        \draw[-] ([xshift=-0.5cm, yshift=0.2cm] 5.1,1.5) -- ([xshift=-0.5cm, yshift=0.7cm] 5.1,-0.5) node[midway,left]{};

         \draw [decorate,decoration={brace,amplitude=7pt,mirror},xshift=-4pt,yshift=0pt] (4.75,2.1) -- (4.75,0.1) node [black,midway,rotate=90,yshift=+0.6cm]{};
        
        
    \end{tikzpicture}
    \vspace{27mm}
    \caption{HAN}
    \label{waca}
\end{algorithm}


\subsection{Gain formula} \label{credit}
To make the network realistic, neurons can learn only thanks to the knowledge of the spikes emitted by the network in the past, and do not have access to spiking probabilities. We use the following gain, computed in Step $12$ of Algorithm \ref{waca}:
\begin{align} \label{formule credit}
    &g^{i^+\to j}_m = \left\{
    \begin{array}{lll}
         \widehat{p^i_m} \times \frac{M}{M^j} &\text{if } o(m)\in j  \\ 
          \\
        - \widehat{p^i_m} \!\times\! \frac{M}{M^{j'}} \!\times\! \frac{1}{\abs{J}-1} &\text{if } o(m) \in j' 
    \end{array}
\right.
\end{align}
where $j'\neq j$, $\widehat{p^i_m}:=\sum_{t=1}^N X_{m,t}^i/N $ and $g^{i^-\to j}_m = -g^{i^+\to j}_m$. Indeed, if the object belongs to class $j$, then the network classifies correctly the object if neuron $j$ spikes more than the others, so excitatory (resp. inhibitory) connections get positive (resp. negative) gains to force $j$ to spike more. Otherwise, $j$ should spike less than the neuron coding for the correct class, so the excitatory (resp. inhibitory) connections get negative (resp. positive) gains (\ie penalties or losses) to get $j$ to spike less. In the present work, the gain $g^{i^\bullet \to j}_m$ depends on the correct class of the presented object, but not on the network own classification, \ie the category in which the network classified the previous objects. Whether the network correctly classifies the object or not does not influence the gain.

Using this gain supposes that we know in advance the value of $M^j$, that is the number of presented objects in class $j$ among the first $M$ objects, and this, for every $j \in J$. If it is not the case, we can replace $\frac{M}{M^j}$ by $\frac{m}{m^j}$ at object $m$, which can be seen as an estimator of the proportion of objects belonging to class $j$.

\subsection{Feasible weight family }
In order to study HAN theoretically, we need to define families of feasible weights, which will be ideal weights that do not vary in time and whose performance we want to match. As stated in section \ref{sec input}, the activity of input neurons depends only on the nature of the presented object, and not on time. Hence, the spiking probability of neuron $j$ with constant synaptic weights $q^j=(q^{i^\bullet \to j})_{i^\bullet\in I^j}$ when object with nature $o\in \mathcal{O}$ is presented to the network does not depend on time either and is $p^j_o(q^j):=\mathbb{E}[p^{j,\text{cond}}_{m,t}(q^j)]$.

\begin{definition}[Feasible weight family]
A \emph{feasible weight family} is a constant weight family $q=(q^j)_{j\in J}$ independent of $m$ such that for all $j\in J$, $o\in j$, $j'\neq j$,
\[p^j_o(q^j) > p^{j'}_o(q^{j'}).\]
The constant
\[\text{Disc}_{\text{safe}}(q) := \min_{j\in J, o\in j, j'\neq j} \Big\{ p^j_o(q^j) - p^{j'}_o(q^{j'})\Big\} \]
is called the \emph{safety discrepancy} of the family $q$. We denote $\mathcal{Q}$ the set of feasible weight families.
\end{definition}

Hence, a feasible weight family is a weights family which enables the network to correctly classify the objects in general; the larger $\text{Disc}_{\text{safe}}(q)$, the lesser it will be mistaken. In Appendix \ref{app lim}, we give examples of such a feasible weight family in a particular case. 

In section \ref{subsec average}, we give theoretical results about HAN Solo when $f$ meets certain conditions; in section \ref{limit}, we study the limit behavior of HAN with the EWA algorithm, and in section \ref{concrete ex} we study a specific case of network and we compare numerically HAN with EWA, HAN with PWA and the Component-Cue model from which HAN is inspired.

\section{THEORETICAL RESULTS} \label{average}

\subsection{Average learning}
\label{subsec average}
In this section, we give theoretical guarantees that our network learns to classify objects as well as any feasible weight family on average under certain conditions. We consider the HAN Solo case $\varphi=\text{Id}$, $\alpha^j=0$ and $I^j_-=\emptyset$. In this simplified framework the set $I^j$ is identified with the set $I^j_+$ so the experts are the input neurons linked to neuron $j$ and we use the notation $i$ instead of $i^+$.

Expert aggregation algorithms are designed to achieve low regret bounds. More precisely, when applied to our setting, the regret of the forecaster/neuron $j$ is 
\[R^j_M:= \max_{q^j\in \mathcal{X}^j} \sum_{i\in I^j} q^{i\to j} G^{i\to j}_M - G^ j_m\]
where $\mathcal{X}^j$ is the set of probability distributions over $I^j$. Note that the maximum is achieved for any combination of diracs on the experts with maximum cumulated gain (there can be ties). This regret can be translated in spiking probabilities of neuron $j$ thanks to our particular choice of gain (see Appendix \ref{app gain}). However, if we want to understand how the network learns, we need a more global notion involving the activity of all output neurons. Let 
$$\widehat{p^j_m(q^j_m) }:= \sum\limits_{i\in I^j} q^{i\to j}_m \widehat{p^i_m}$$
where $q^j_{1:M}:=(q^j_m)_{1\leq m\leq M}$ is a weight family. This an estimator of the spiking probability of neuron $j$ if the synaptic weights were given by $q^j_m$ during the presentation of the $m^{th}$ object. For any weights $q^j_{1:M}:=(q^j_m)_{1\leq m\leq M}$, we then interpret
 $$\widehat{P^{j,j'}_M(q^j_{1:M})} := \dfrac{1}{M^{j'}}\sum\limits_{m, \text{ }o(m)\in j'}
  \widehat{p^j_m(q^j_m)}$$ as an estimator of the average spiking probability of neuron $j$ with weights $q^j_{1:M}$ during the presentation of objects in class $j'$. In the notation $P^{j,j'}_M$, index $j$ refers to a neuron, whereas index $j'$ refers to a class.

Then the \emph{class discrepancy} of class $j$ is for a network governed by weights $q_{1:M}:=(q^j_{1:M})_{j\in J}$ is defined by $$\text{Disc}^j_M(q_{1:M}) = \widehat{P^{j,j}_M(q^j_{1:M})} - \dfrac{1}{\abs{J}-1}\sum\limits_{j'\neq j} \widehat{P^{j',j}_M(q^{j'}_{1:M})}.$$
  It measures how much neuron $j$ fires more than the other neurons when an object of class $j$ is shown, and it is therefore a global information at the network level. 
  We give another choice of gain (with some drawbacks) in Appendix \ref{app gain}, where the class discrepancy can be expressed in terms of empirical spiking probabilities.

Finally, the average class discrepancy of output neurons is called the \textit{network discrepancy} and is defined by $$\text{Disc}_M(q_{1:M}) = \frac{1}{\abs{J}}\sum_{j\in J}\text{Disc}^j_M(q_{1:M}).$$
Like safety discrepancy for a feasible weight family, the network discrepancy measures how much output neurons fire more than the others when an object of their class is shown. However, unlike safety discrepancy, it gives average information rather than quantifying the worst possible deviation.

\begin{assumption}[Regret bound] \label{assumption reg}
The expert aggregation algorithm $f$ used in HAN Solo is such that for any deterministic sequence of gains $g^j_{1:M}\in [a,b]$, 
\[R^j_M \leq K(\abs{I},b-a) \sqrt{M}\]
where $K(\abs{I},b-a)$ is a constant depending on $\abs{I}$ and $b-a$.
\end{assumption}
Both EWA \eqref{EWA} and PWA \eqref{PWA} satisfy Assumption \ref{assumption reg}. See Appendix \ref{app regret bounds} for details on the bounds. 

\begin{assumption} \label{assump proportions}
    There exists a constant $\xi>0$ independent of $M$ such that for every $j\in J$, $M_j/M\geq \xi$.
\end{assumption}
Assumption \ref{assump proportions} means that every class of objects is well represented during the learning phase.
\begin{theorem}[Oracle inequality] \label{theorem oracle}
Suppose Assumptions \ref{assumption reg} and \ref{assump proportions} hold. Let $\alpha \in (0,1]$. Suppose $\mathcal{Q}$ is non-empty. Then with probability greater than $1-\alpha$,
 $$ \text{Disc}_M(w_{1:M}) \geq \max_{q\in \mathcal{Q}}\text{Disc}_{\text{safe}}(q) - E_{\text{tot}}(N,M,\alpha)$$
 where $E_{\text{tot}}(N,M,\alpha) = E_{\text{reg}}(M) + E(N,M,\alpha)$ with
$$
    E_{\text{reg}}(M) :=K\Big(\abs{I},\frac{\abs{J}}{\xi(\abs{J}-1)}\Big)\frac{1}{\sqrt{M}}$$
and
    $$E(N,M,\alpha):=\sqrt{\ln\Big(\frac{2\abs{I}\abs{J}}{\alpha}\Big)\frac{2}{\xi NM}}.$$

\end{theorem}
In a nutshell, assuming the error $E_{\text{tot}}(N,M,\alpha)$ is negligible, this result on network discrepancy means that in average, a class neuron spikes more than the other neurons when presented with an object in its class with high probability. Thus, in average, the network correctly classifies the objects under the hypothesis that a feasible weight family  exists. 
The error is twofold: one part, $E_{\text{reg}}(M)$, comes from the regret bound and is in $O(M^{-1/2})$, whatever $N$. The other part, $E(N,M,\alpha)$, comes from the inherent randomness of our system and is in is in $O((NM)^{-1/2})$. Hence if $M$ is large enough, the total error $E_{\text{tot}}(N,M,\alpha)$ is negligible compared to the constant $\max_{q\in \mathcal{Q}}\text{Disc}_{\text{safe}}(q)$. In this sense, HAN performs as well as an oracle that would know the best feasible weight family in advance: its network discrepancy is larger than the best safety discrepancy, with asymptotic multiplicative constant 1.  However if $M$ is not large enough we pay a price in $O(M^{-1/2})$ for having seen only that many objects and being initialized with a weight family that is not feasible. Finally if $M$ and $N$ are not large enough, the randomness in the system increases, and the approximation of the activity of input neurons given by the gains may be insufficient to find the best experts among them, or the time of the presentation of an object could be too short to see which output neuron significantly spikes the most.

\subsection{Limit behavior} 
\label{limit}

In this section, independent of the previous one, we are going to study the limit behavior of the network. We are no longer interested in average results: the linearity of $\varphi$ is not needed anymore and \emph{we are in the general case \eqref{Hawkes}, where inhibition is allowed}. However, we want to conduct a more precise analysis of the network's limit behavior and this analysis can only be carried out on a case-by-case basis, depending on the expert aggregation algorithm chosen. Here, we have decided to use the EWA algorithm, for its simplicity and universality.

Instead of assuming that a feasible weight family exists as in the previous section, we want to build directly the limit of the weights, hoping that this limit makes sense from a learning point of view. But if the input neurons encoding the features have nothing to do with the output class (e.g. two classes "blue" and "red" and all neurons having the same firing rates whatever the color) the problem cannot be solved. This is why we introduce the notion of \emph{feature discrepancy} of connection $i^+$ (resp. $i^-$) with respect to class $j$, defined by: 
\[d^{i^+\to j} := \frac{1}{n^j}\sum_{o\in j} p^i_o - \frac{1}{\abs{J}-1} \sum_{j'\neq j} \frac{1}{n^{j'}}\sum_{o\in j'}p^i_o\]
(resp. $d^{i^-\to j}=-d^{i^+\to j}$) where $n^j$ is the number of natures of objects belonging to class $j$.
For excitatory connections (resp. inhibitory), this feature discrepancy is the difference between the average firing rate of neuron $i$ when presented objects belonging to class $j$, and the average firing rate of neuron $i$ when presented objects belonging to other classes (resp. the opposite). It indicates the extent to which neuron $i$ has higher-than-usual firing rate when presented with objects in category $j$. Thus one can define the set of connections that are the most sensible to class $j$: $\tilde{I}^j=\arg\max_{i^\bullet\in I^j} d^{i^\bullet \to j},$ as well as the gap in discrepancy if $\tilde{I}^j\not =I^j$:
$$\gamma^j= \max_{i^\bullet\in I^j} d^{i^\bullet \to j} - \max_{i^\bullet\in I^j\setminus \tilde{I}^j} d^{i^\bullet \to j},$$
which measures how good  the most sensible connections are with respect to the others. Note in particular that if all the $d^{i^\bullet \to j}$'s are null, $\tilde{I}^j=I^j$, there is no gap and nothing can be learned from the network because the input neurons are in fact not sensible to important features for the classification.

\begin{theorem} \label{theorem cvg feasible weight family} 
Suppose each nature of object is presented the same amount of times: for all $o\in \mathcal{O}$, $\abs{\{m, o(m)=o\}}=\frac{M}{\abs{\mathcal{O}}}$. Let $\eta^j=\frac{1}{\abs{\mathcal{O}}}\sqrt{2\frac{\ln(\abs{I^j})}{M}}$ and $w_\infty:= (w_\infty^{i^\bullet \to j})_{j\in J, i^\bullet\in I^j}$ where
$$ w_\infty^{i^\bullet \to j} = \left\{
    \begin{array}{ll}
          \abs{\Tilde{I}^j}^{-1} &\text{if } i^\bullet\in \Tilde{I}^j \\
       0 & \text{otherwise}.
    \end{array}
    \right.
$$
Then with probability $1-\alpha$, for all $j\in J$:
\\$\bullet$ if $\tilde{I}^j\neq I^j$, then the weights $w^{i^\bullet \to j}_{M+1}$ at the end of the learning phase satisfy
$$|w^{i^\bullet \to j}_{M+1}-w_\infty^{i^\bullet \to j}|\leq E^j(N,\alpha) + E^j_{\text{EWA}}(M)$$
where 
$$E^j_{\text{EWA}}(M) :=\max\Big\{1, \frac{\abs{I^j}}{\abs{\Tilde{I}^j}}-1\Big\}\frac{1}{\abs{\Tilde{I}^j}}e^{-\frac{\gamma^j}{\abs{\mathcal{O}}}\sqrt{2\ln(\abs{I^j})M}}$$ 
and 
$$E^j(N,\alpha) :=   \abs{I^j}\sqrt{\ln\Big(\frac{2\abs{I}\abs{J}}{\alpha}\Big)\frac{\ln(\abs{I^j})}{\abs{\mathcal{O}}N}}.$$ \\
$\bullet$ if $\tilde{I}^j=I^j$, then the weights $w^{i^\bullet \to j}_{M+1}$ at the end of the learning phase mostly did not evolve, i.e. 
$$|w^{ i^\bullet \to j}_{M+1} - \abs{I^j}^{-1}| \leq E^j(N,\alpha).$$ 
\end{theorem}
This choice of $\eta^j$ has good theoretical guarantees (if the time horizon is unknown, we can use a similar time-dependent learning rate providing similar results, see details in Appendix \ref{app regret bounds}). Note that this result applies for weights $w_{M+1}^j$ and not $w_m^j$: it is true only at the end of the learning phase. In the first case, the  synaptic weights converge to the constant family $w_\infty$, which is uniform on input neurons with maximal feature discrepancy. The error is twofold. One part, $E^j(N,\alpha)$, coming from the randomness of our system is in $O(N^{-1/2})$ and is the equivalent of the error term $E(N,M,\alpha)$ in the oracle inequality (Theorem \ref{theorem oracle}). The other part, $E^j_{EWA}(M)$, coming from the fact that we have only seen $M$ objects and very specific to the EWA algorithm is in $\exp(-O(\gamma^j\sqrt{M}))$ and is the equivalent of the error term $E_{\text{reg}}(M)$ which was specific to the expert aggregation algorithm used in HAN Solo. This points out that the larger the gap in discrepancy the quicker the learning. The second case is the non interesting one, where there is nothing to learn (for large $N$ weights are back to their initial value) either because HAN with EWA is good since initialization or because the features are so badly encoded by the input neurons that nothing can be learned.

 Going from the local notion of feature discrepancy to the global notion of network discrepancy is not straightforward, even if it seems intuitive to hope that the weight family with the largest feature discrepancy also achieves the largest network discrepancy. This is why, to complete the circle, we need to assume that $w_\infty$ is a feasible weight family in the next corollary.

\begin{corollary}\label{cor}
Suppose the assumptions of Theorem \ref{theorem cvg feasible weight family} hold and $w_\infty$ is a feasible weight family. Let $L$ be the Lipschitz constant of $\varphi$. Then at the end of the learning phase, with probability $1-\alpha$, for all $j\in J$, $j'\neq j$, $t\in \{1,\dots,N\}$, supposing that we present an object $o(M+1)\in j$ to the network we have
\begin{align*}
    &p^{j,\text{cond}}_{M+1,t}(w^j_{M+1}) -p^{j',\text{cond}}_{M+1,t}(w^{j'}_{M+1}) \\
    &\geq \Big( p^{j,\text{cond}}_{M+1,t}(w^j_{\infty}) -p^{j',\text{cond}}_{M+1,t}(w^{j'}_{\infty}) \Big) - E^{j,j'}_{\text{tot}}(N,M,\alpha)
\end{align*}
where 
\[
E^{j,j'}_{\text{tot}}(N,M,\alpha):= L\sum_{h\in \{j,j'\}} \abs{I^h}(E^h(N,\alpha) + E^h_{\text{EWA}}(M)).
\]
\end{corollary}
Note that $\mathbb{E}[p^{j,\text{cond}}_{M+1,t}(w^j_{\infty}) -p^{j',\text{cond}}_{M+1,t}(w^{j'}_{\infty})] \geq \text{Disc}_{\text{safe}}(q)$. This means that with high probability, at the end of the learning phase the network classifies objects as well as it would with the feasible weight family $w_\infty$, with a decreasing error term of the same order as in Theorem \ref{theorem cvg feasible weight family}. This corollary can be seen as a non-average version of the oracle inequality of Theorem \ref{theorem oracle}.

\section{A CONCRETE EXAMPLE} \label{concrete ex}
 In this section, we give a specific case for which the limit can be guessed beforehand and is a feasible weight family under certain conditions in both HAN and HAN Solo framework, and we compare numerically HAN with EWA, with PWA and the Component-Cue model \citep{gluck1988conditioning}.

 \subsection{Framework} \label{framework ex}

 \begin{figure*}[ht]
    \centering
    \begin{minipage}{.4\textwidth}
        \vspace{-0.1cm}
        \hspace{-1.5cm}
        \includegraphics[width=8cm]{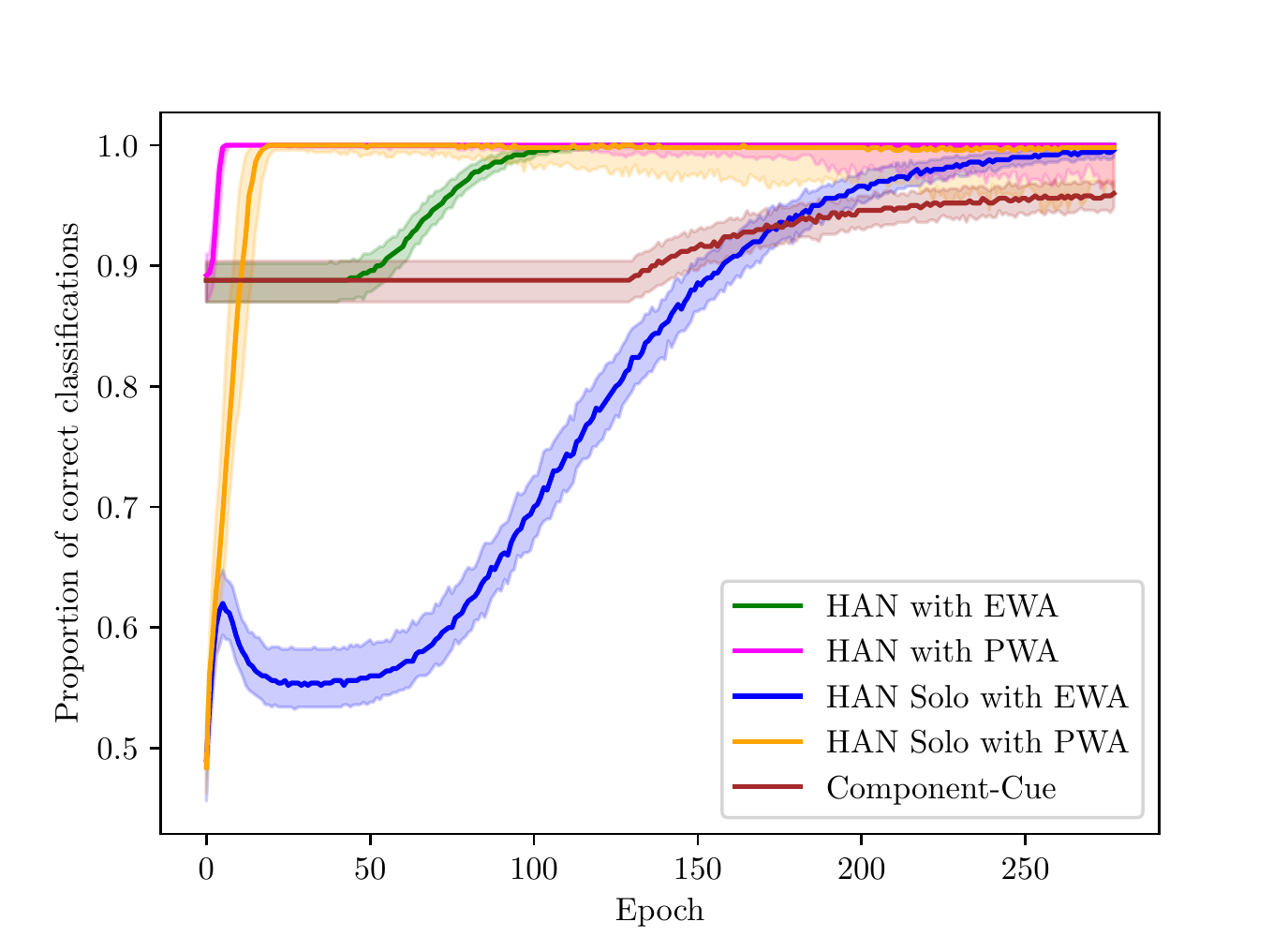}
    \end{minipage}%
    \raisebox{1ex}
    {\begin{minipage}{0.4\textwidth}
        \vspace{0.1cm}
        \hspace{-0.2cm}
        \includegraphics[width=8cm]{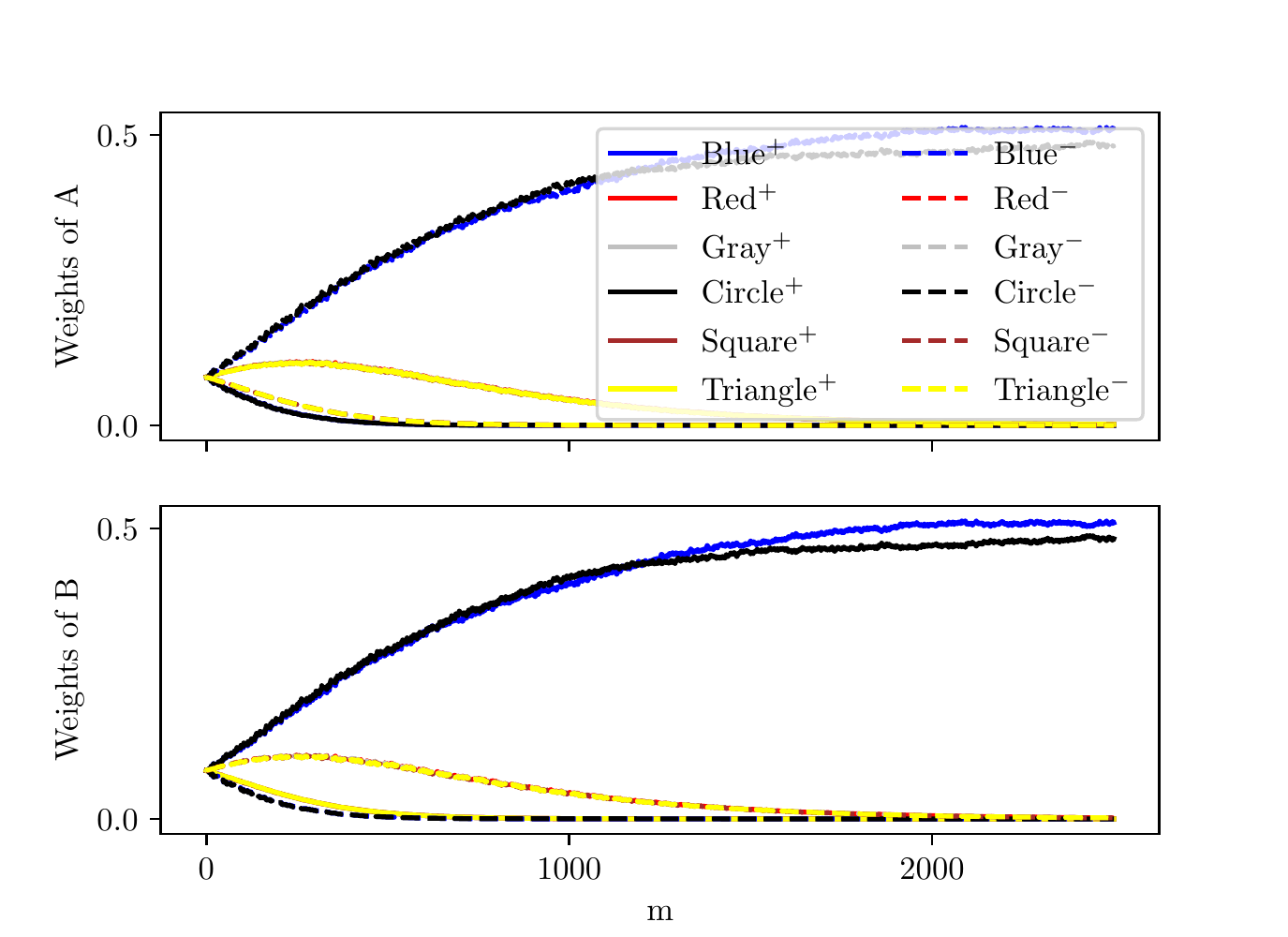}
    \end{minipage}}
    
    \caption{Numerical results with $M=2502$, $K=1$, $N=1000$, $p=0.2$, $q=0.3$, $\alpha^A=0.2$, $\alpha^B=0$, $\beta^j=2$ and $\eta^j=\frac{1}{\abs{\mathcal{O}}}(2\frac{\ln(\abs{I^j})}{M})^{-1/2}$. Parameters of Component-Cue: $\lambda_w=0.005$ and $\phi=10$ (see details in Appendix \ref{app cc}). On the left, evolution of the proportion of correct classifications for HAN and HAN Solo with EWA and PWA and Component-Cue with time. A number $100$ of realizations were made; for each realization, a testing set of $500$ objects drawn randomly was generated. Then the network was trained for $278$ epochs, an epoch being a random sequence of the $9$ nature of objects. After each epoch the weights were frozen and the network performance was evaluated on the testing set. On the $x$-axis, number of epochs. On the $y$-axis, proportion of correctly classified objects of the testing set with confidence interval of level $0.9$. On the right, evolution of the weights of neurons $A$ and $B$ with time for one realization of HAN with EWA.}
    \label{curves}
\end{figure*}

In this section, we assume that the processes emitted by input neurons are mutually independent. There are $c$ characteristics, each declined in $n$ features. Therefore, a nature of object is
identified with $c$ given features, one for each characteristic, and there are $n$ 
choices for each feature. We consider two classes, class $B$, containing one nature of object, 
and class $A$, containing every possible other nature of object. Class $B$ represents an exception.
The features are denoted by $f_{k,l}$, where $k\in \{1,\dots, c\}$ and $l\in \{1,\dots, n\}$, and for each feature $f_{k,l}$, there is one input neuron, 
  also denoted $f_{k,l}$, which spikes with probability $p$ when presented with an 
  object having the feature $f_{k,l}$. Each nature of object is presented the same amount of times to the network. \\
$\bullet$ \textbf{HAN}: we study the case  $I^j_+=I^j_-=I$, $\varphi = (\cdot)_+ \wedge 1$, 
$\alpha^B=0$ and $K=1$. Then under some assumptions the limit of the weights, computed in Appendix \ref{specific HAN},
 is a feasible weight family and the network correctly classify the objects asymptotically. \\
$\bullet$ \textbf{HAN Solo}: we study the case $I^j_+=I$, $I^j_-=\emptyset$, $\varphi = \text{Id}$, 
$\alpha^j=0$. To replace inhibition, we add input neurons to the network: for each feature $f_{k,l}$,
we add a neuron $\Tilde{f}_{k,l}$ which spikes with probability $q$ when presented with an object which does not 
have feature $f_{k,l}$. Hence, neuron $f_{k,l}$ detects the presence of feature $f_{k,l}$, 
and neuron $\Tilde{f}_{k,l}$ detects its absence. Then under some assumptions the limit of the weights, computed in Appendix \ref{specific HAN Solo},
is a feasible weight family and the network correctly classify the objects on average and asymptotically. 

\subsection{Numerical results} \label{numerical results}

To illustrate this specific case, we use $c=2$ characteristics with $n=3$ features for each: the shape, corresponding to the features circle, square and triangle, and the color, corresponding to the features blue, gray and red. The classes are $A=\{{\color{blue}\Box}, {\color{blue} \triangle}, {\color{gray} \bigcirc}, {\color{gray} \Box}, {\color{gray} \triangle}, {\color{red} \bigcirc}, {\color{red} \Box}, {\color{red} \triangle}\}$ and $B=\{{\color{blue} \bigcirc}\}$.

\begin{table*}[h]
\caption{Ablation study} \label{ablation}
\begin{center}

\begin{tabular}{|c|c|c|c|c|} 
 \hline
 Ablated features & HAN EWA & HAN PWA & HAN Solo EWA & HAN Solo PWA\\ 
 \hline
$0/6$ &  $99.9 $& $99.5 $& $99.4 $ & $98.6 $ \\ 
$1/6$ &  $93.0 $& $92.1 $& $92.8 $& $90.1 $\\ 
$2/6$ &  $88.6 $ & $87.1 $& $82.2 $& $81.6 $\\ 
$3/6$ &  $83.4 $ & $85.2 $& $74.8 $& $68.6 $\\ 
$4/6$ &  $84.5 $ & $84.2 $& $58.7 $& $51.9 $\\ 
$5/6$ &  $84.9 $ & $85.1 $& $55.8 $& $52.5 $\\ 
 \hline
\end{tabular} 

\end{center}
\end{table*}

The evolution of the proportion of correct classifications of HAN and HAN Solo with the setting of section \ref{framework ex} for both EWA and PWA and Component-Cue are visible on the left of Figure \ref{curves} (for a description of Component-Cue and its parameters see Appendix \ref{app cc}). For both PWA and EWA, HAN reaches perfect performance faster than HAN Solo: the use of non-linear $\varphi$ and inhibition was more effective than the addition of neurons coding for the absence of features. Besides, for both HAN and HAN Solo, PWA learns faster than EWA but its variance is higher and seems to grow in time, whereas the one of EWA decays. Component-Cue is the only one which does not achieve perfect performance at the end of the learning phase, and its variance does not seem to evolve.

 The evolution of the weights for one realization of HAN with EWA is visible on the left of Figure \ref{curves}, which illustrates Theorem \ref{theorem cvg feasible weight family}: the weights of $A$ converge to uniform distribution on connections Blue$^-$ and Circle$^-$, and the weights of $B$ converge to uniform distribution on connections Blue$^+$ and Circle$^+$, which is a feasible weight family. Hence $A$ is inhibited by neurons coding for features of ${\color{blue} \bigcirc}$, which is the object of class $B$, whereas $B$ is excited by these same neurons.

In Table \ref{ablation}, we provide an ablation study to see how HAN and HAN Solo perform with missing input neurons. The Table gives the percentage of correct classifications for each model at the end of the learning phase depending on the number of ablated features, with the same parameters as in Figure \ref{curves}. One ablated feature corresponds to one ablated neuron (resp. two ablated neurons) for HAN (resp. HAN Solo): if feature $f_{k,l}$ is ablated, then neuron $f_{k,l}$ is ablated for HAN and neurons $f_{k,l}$ and $\Tilde{f}_{k,l}$ are ablated for HAN Solo. We can see that the network performance is comparable when using either EWA or PWA for both HAN and HAN Solo. However, HAN seems to perform better than HAN Solo with ablated features for both EWA and PWA.

Figure \ref{percep} shows a comparison with the well-known perceptron learning algorithm. The performance of the perceptron is comparable to that of HAN with PWA. It should be noted that the comparison is not fair since the perceptron does not involve spikes and is designed for performance, whereas HAN is designed to be cognitively relevant.

\section{CONCLUSION}
In this paper, we introduced a Hawkes network that provably learns to classify objects thanks to a local learning rule using an expert aggregation method. 
The main point of our paper is to rigorously prove why our Hawkes network learns. Indeed, our learning rule led to an algorithm (HAN) allowing us to prove an oracle inequality on the network discrepancy in the case of linear Hawkes process, and even limits and rates of convergence in the general case for a specific expert aggregation algorithm. 
A promising -- but ambitious -- line of research is to understand if such local rules can be generalized for Hawkes network with one, or more, hidden layers. A first step in this direction could be to add an intermediate layer with neurons detecting correlations in the activity of feature neurons.
Another line of research could be to try to prove similar regret bounds for STDP, three-factors or rate-based learning rules.

\begin{figure} 
   \hspace{-0.5cm}
    \centering
    \includegraphics[width=7cm]{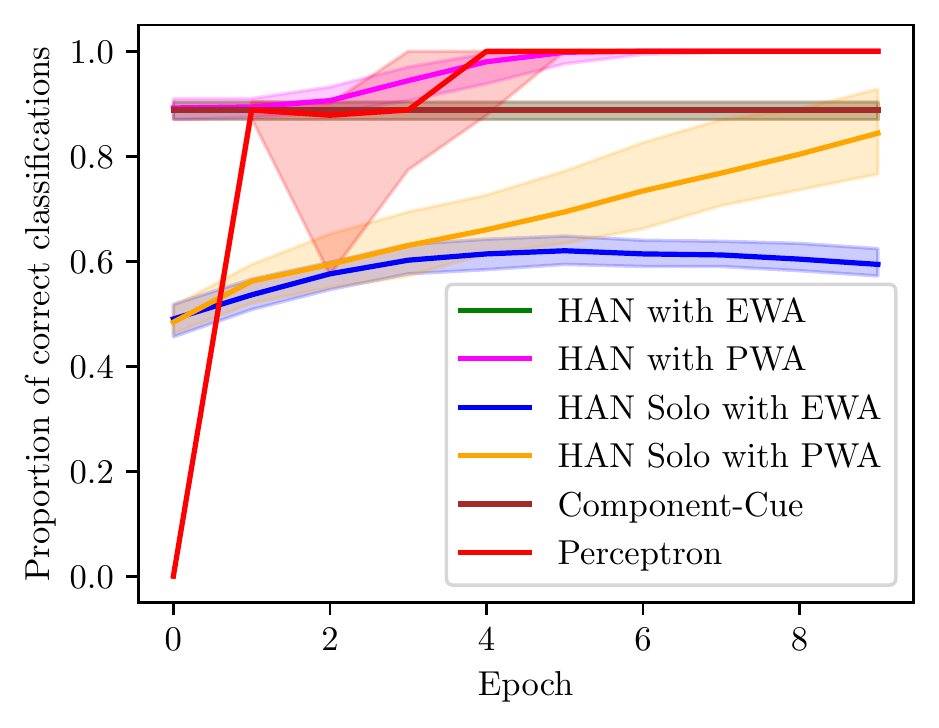}
    \caption{Comparison with the perceptron learning algorithm. Same parameters as in Figure \ref{curves}, learning rate $1$ for the perceptron. Zoom on the $10$ first epochs.}
    \label{percep}
\end{figure}


\section*{Acknowledgment} 
This research was supported by the French government, through CNRS (eXplAIn team), the UCA$^{Jedi}$ and 3iA Côte d'Azur Investissements d'Avenir managed by the National Research Agency (ANR-15 IDEX-01 and ANR-19-P3IA-0002), directly by the ANR project ChaMaNe (ANR-19-CE40-0024-02) and GraVa (ANR-18-CE40-0005), and finally by the interdisciplinary Institute for Modeling in Neuroscience and Cognition (NeuroMod).

\bibliography{papier}

\section*{Checklist}


 \begin{enumerate}

 \item For all models and algorithms presented, check if you include:
 \begin{enumerate}
   \item A clear description of the mathematical setting, assumptions, algorithm, and/or model. \textbf{Yes.}
   \item An analysis of the properties and complexity (time, space, sample size) of any algorithm. \textbf{Yes.} HAN/HAN Solo's complexity is directly linked to the number of rv simulation calls.
   \item (Optional) Anonymized source code, with specification of all dependencies, including external libraries. \textbf{Yes.}
 \end{enumerate}

 \item For any theoretical claim, check if you include:
 \begin{enumerate}
   \item Statements of the full set of assumptions of all theoretical results. \textbf{Yes.}
   \item Complete proofs of all theoretical results. \textbf{Yes.}
   \item Clear explanations of any assumptions. \textbf{Yes.}    
 \end{enumerate}

 \item For all figures and tables that present empirical results, check if you include:
 \begin{enumerate}
   \item The code, data, and instructions needed to reproduce the main experimental results (either in the supplemental material or as a URL). \textbf{Yes.}
   \item All the training details (e.g., data splits, hyperparameters, how they were chosen). \textbf{Yes.}
         \item A clear definition of the specific measure or statistics and error bars (e.g., with respect to the random seed after running experiments multiple times). \textbf{Yes.}
         \item A description of the computing infrastructure used. (e.g., type of GPUs, internal cluster, or cloud provider). \textbf{Yes.} Personal computer without GPU (no timing)
 \end{enumerate}

 \item If you are using existing assets (e.g., code, data, models) or curating/releasing new assets, check if you include:
 \begin{enumerate}
   \item Citations of the creator If your work uses existing assets. \textbf{N/A.}
   \item The license information of the assets, if applicable. \textbf{N/A.}
   \item New assets either in the supplemental material or as a URL, if applicable. \textbf{N/A.}
   \item Information about consent from data providers/curators. \textbf{N/A.}
   \item Discussion of sensible content if applicable, e.g., personally identifiable information or offensive content. \textbf{N/A.}
 \end{enumerate}

 \item If you used crowdsourcing or conducted research with human subjects, check if you include:
 \begin{enumerate}
   \item The full text of instructions given to participants and screenshots. \textbf{N/A.}
   \item Descriptions of potential participant risks, with links to Institutional Review Board (IRB) approvals if applicable. \textbf{N/A.}
   \item The estimated hourly wage paid to participants and the total amount spent on participant compensation. \textbf{N/A.}
 \end{enumerate}

 \end{enumerate}

\newpage

\onecolumn
\aistatstitle{Appendix for “Provable local learning rule by expert aggregation for a Hawkes network''}

Appendix \ref{app notations} provides a table of notations, Appendix \ref{app expe} $-$ \ref{app gain} provide additional theoretical and numerical results, and Appendix \ref{proofs} provides the proofs of our results.

\section{TABLE OF NOTATIONS} \label{app notations}

All the notations are listed in Table \ref{tab_notations}.

\begin{table}[h]
\caption{Notations} \label{tab_notations}
\begin{center}
\begin{tabular}{ll}
\textbf{NOTATION}  &\textbf{DESCRIPTION} \\
\hline \\
 $\mathcal{O}$ & set of objects \\
    
     $M$ & total number of objects presented to the network \\
    
     $N$ & number of time steps during which one object is presented \\
   
     $m$ & index of the current round \\
   
     $o$ & object in $\mathcal{O}$ \\
   
     $o(m)$ & nature of the object of the $m^{th}$ round \\
    
    $J$ & set of classes and of output neurons \\
   
    $I$ & set of input neurons \\
    
    $j$ & index of an output neuron and of a class \\
    
    $M^j$ & number of rounds with objects belonging to class $j$ \\
    
    $i$ & index of an input neuron \\
   
    $I^j_+$ & set of excitatory input neurons of output neuron $j$ \\
  
    $I^j_-$ & set of inhibitory input neurons of output neuron $j$ \\
  
$i^\bullet$ & connection $i^\bullet\in \{i^+,i^-\}$ \\
   
     $I^{+/-}$ & set of signed input neurons, \ie all possible connections to output neurons \\
   
 $I^j$ & set of connections to output neuron $j$ \\
 
     $\varphi$ & activation function of the Hawkes process \\
   
     $K$ & size of the support of functions $g_+$ and $g_-$ \\
   
     $\alpha^j$ & spontaneous activity of output neuron $j$ \\
   
    $X_{m,t}^i$ & activity of neuron $i$ during time step $t$ of round $m$. \\
   
    $p^i_m$ & spiking probability of input neuron $i$ during round $m$ \\
   
    $\mathcal{Q}$ & set of feasible weight families  \\
   
    $p^i_o$ & spiking probability of input neuron $i$ when presented with object $o$    \\
   
    $p^{j,\text{cond}}_{m,t}$ & conditional spiking probability of output neuron $j$ knowing the past \\
    
    $w^{i^\bullet\to j}_m$ & synaptic weight of connection $i^\bullet$ of neuron $j$ during round $m$ given by HAN \\
    
    $w_m^j$ & weight family of neuron $j$ during round $m$ given by HAN: $(w^{i\to j}_m)_{i\in I^j}$ \\
   
    $q^{i^\bullet\to j}$ & constant synaptic weight of connection $i^\bullet$ of $j$ \\
    
    $q^j$ & constant weight family of neuron $j$: $(q^{i^\bullet\to j})_{i^\bullet\in I_j}$ \\
   
     $g^{i^\bullet\to j}_m$ & gain of connection $i^\bullet$ w.r.t. output neuron $j$ for round $m$ \\
   
     $G^{i^\bullet\to j}_m$ & cumulated gain of connection $i^\bullet$ of output neuron $j$ until round $m$ \\
    
     $G_j^m$ & cumulated gain of output neuron $j$ until round $m$: $\sum_{m'=1}^m \sum_{i^\bullet\in I^j} w^{i^\bullet\to j}_m g^{i^\bullet\to j}_{m'}$ \\
     
     $\eta^j$ & learning rate of EWA \\
   
     $\beta^j$ & parameter of PWA \\
   
    $A$, $B$ & classes of the specific case \\
   
    $c$ & number of characteristics in the specific case \\
   
    $n$ & number of features for each characteristic in the specific case \\
   
    $o^B$ & unique object belonging to class $B$ \\
   
    $f_{k,l}$ & feature of the specific case \\
   
    $f_{k,l}, \Tilde{f}_{k,l}$ & neurons of the specific case \\
   
    $p$ & spiking probability of $f_{k,l}$ neurons when active \\
   
    $q$ & spiking probability of $ \Tilde{f}_{k,l}$ neurons when active \\
\end{tabular}
\end{center}
\end{table}

\section{ADDITIONAL EXPERIMENTS} \label{app expe}

On Figure \ref{fig prob}, the evolution of the empirical spiking probabilities (\ie $\frac{N^A_m}{N}$ and $\frac{N^B_m}{N}$) with time for each nature of object is visible, for the realization of HAN Solo with EWA illustrated on the left of Figure \ref{curves}.

 We can see that neuron $B$ does not spike at all when presented with objects having no feature in common with the blue circle (\ie the red square, red triangle, gray square and gray triangle), and spikes less than $A$ when presented with objects having one feature in common (\ie the blue square, blue triangle, red circle and gray circle). At the beginning, $B$ spikes less than $A$ when presented with the blue circle and after some time it starts spiking more. This explains why on the right of Figure \ref{curves}, the curve of HAN with EWA is constant at around $8/9$ ($8$ of the $9$ natures of object are well classified) and after some time rises to $1$.

On Figure \ref{fig random}, we can see the evolution of the proportion of correct classifications for HAN and HAN Solo
with EWA and PWA and Component-Cue with the same parameters as in Figure \ref{curves}, but here
the objects are randomly selected with replacement so the network is not guaranteed
 to see all natures of objects in a given epoch. All the variances are increased compared to Figure \ref{curves}; however, the one of HAN with EWA decreases after some time and HAN with EWA is
 the only algorithm reaching almost perfect performance. Although HAN Solo with EWA 
 has a high variance, its performance is improving and could reach higher values with more time, but the curves of the other algorithms do not seem to be converging towards $1$.

 \begin{figure}[ht]
    \centering
    \includegraphics[width=10cm]{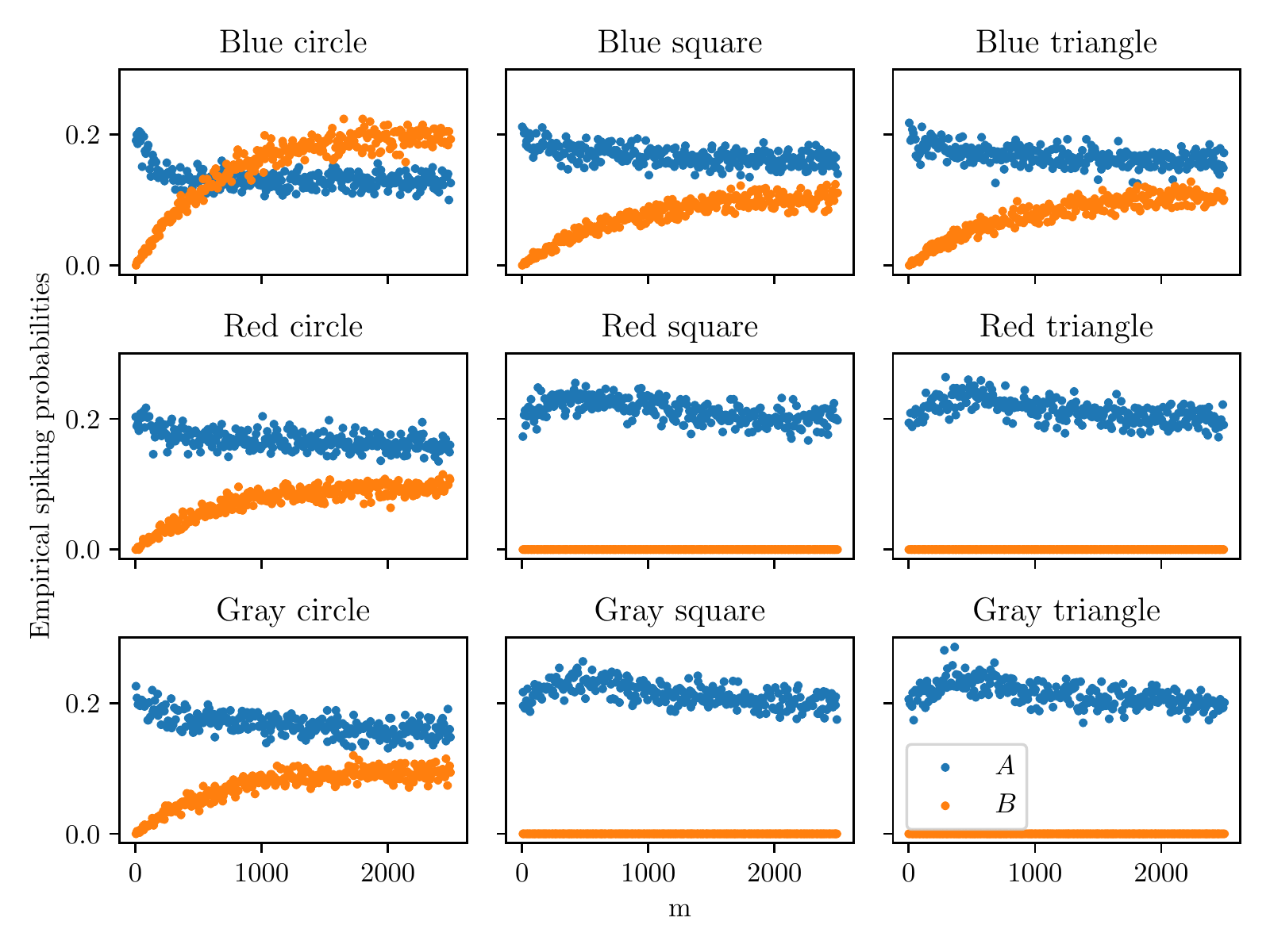}
    \caption{Evolution of the empirical spiking probabilities of neurons $A$ and $B$ with time by nature of object for the same realization of HAN with EWA as in Figure \ref{curves} (left). (Same parameters as in Figure \ref{curves}.) }
    \label{fig prob}
\end{figure}

\begin{figure}[ht]
    \centering
    \includegraphics[width=10cm]{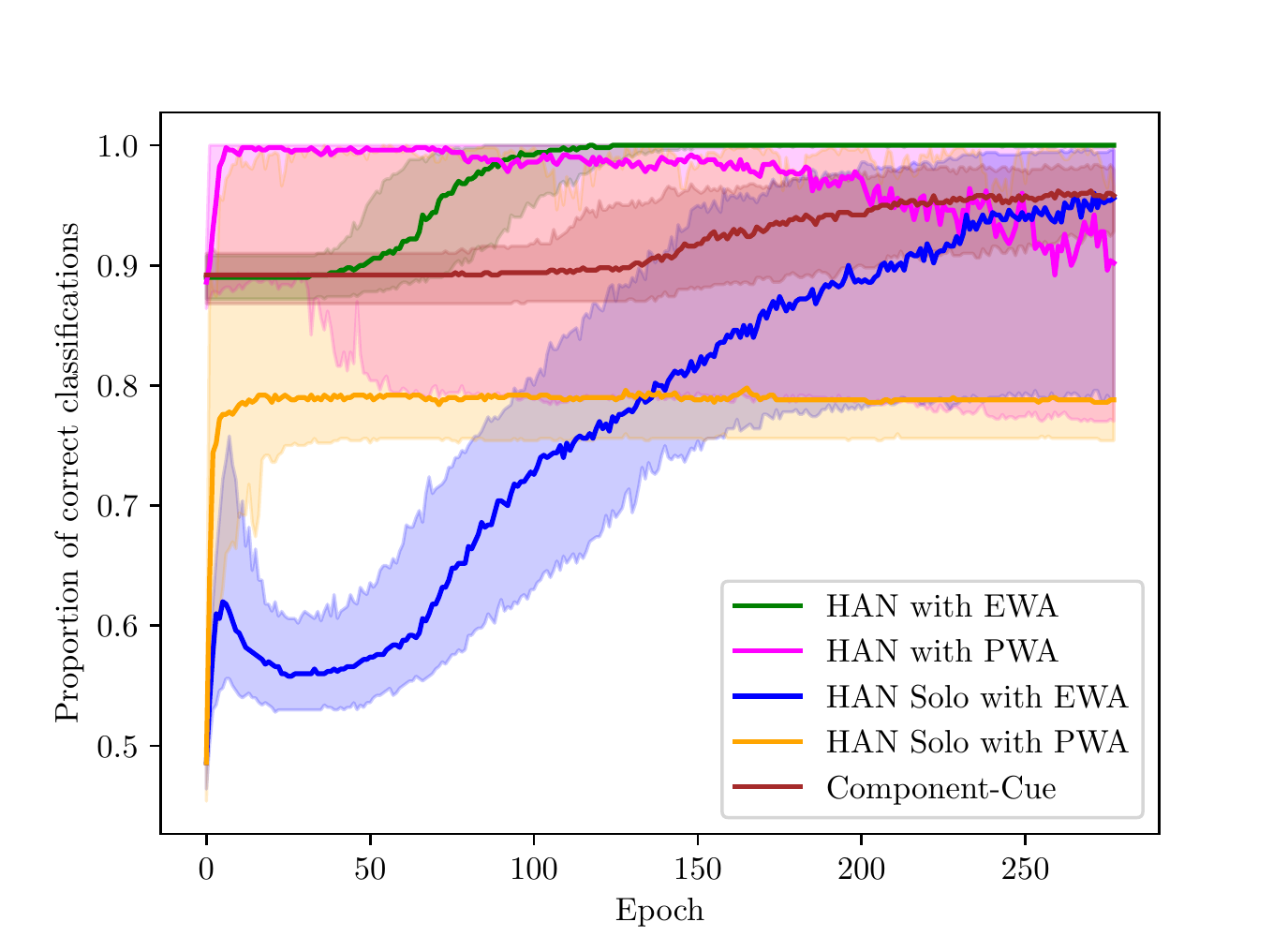}
    \caption{Evolution of the proportion of correct classifications for HAN and HAN Solo
     with EWA and PWA and Component-Cue with time, with the same parameters as in Figure 
     $2$. What changes here is that an epoch is a sequence of $9$ objects drawn randomly 
     with replacement: all the natures of object are not necessarily presented during one
      epoch.}
      \label{fig random}
\end{figure}

\section{DETAILS ABOUT THE LIMIT WEIGHTS OF SECTION \ref{framework ex}} \label{app lim}

\subsection{Study of HAN with EWA} \label{specific HAN}

 Here  $I^j_+=I^j_-=I$, $\varphi = (\cdot)_+ \wedge 1$, $\alpha^B=0$ and $K=1$.

\begin{proposition} \label{prop specific han}
    Suppose each nature of object is presented the same amount of times, $n>2$, $(c-1)p < (1-p)^{c-1}$. Then the conditions of Theorem \ref{theorem cvg feasible weight family} are verified and there exists $\alpha^A>0$ such that the limit weights $w_\infty:=(w_\infty^A,w_\infty^B) $ are a feasible weight family such that $w_\infty^A$ puts the weight $c^{-1}$ on every connection $f_{k,1}^-$ and $w_\infty^B$ puts the weight $c^{-1}$ on every connection $f_{k,1}^+$. Besides, 
    \[\text{Disc}_{\text{safe}}(w_\infty) = \min\Big\{\alpha^A(1-p)^{c-1} - \frac{c-1}{c}p, p-\alpha^A(1-p)^c\Big\}.\]
\end{proposition}
Note that $w_\infty^A$ uniformly distributes weight on inhibitory connections to neurons active when presented with $o_B$, while $w_\infty^B$ uniformly distributes weight on excitatory connections to neurons active when presented with $o^B$, so it is easy to see why it is a feasible weight family. Hence, under the assumptions of Proposition \ref{prop specific han solo}, the conclusion of Corollary \ref{cor} holds: the network correctly classifies the objects asymptotically. 

We can see an illustration of this proposition on Figure \ref{curves}. 

\subsection{Study of HAN Solo with EWA} \label{specific HAN Solo}

Here $\varphi=\text{Id}$, $\alpha^j=0$, $I^j_-=\emptyset$ and $I^j=I$. 

\begin{proposition} \label{prop specific han solo} Suppose each nature of object is presented the same amount of times, $n\geq 2$,
$p (n-1)^{-1} < q < p (n-1)$ and $q > (c-1)p$.
Then the conditions of Theorem \ref{theorem cvg feasible weight family} are verified and the limit weights $w_\infty:=(w_\infty^A,w_\infty^B)$ are a feasible weight family such that $w_\infty^A$ puts the weight $c^{-1}$ on every neurons $\Tilde{f}_{k,1}$ and $w_\infty^B$ puts the weight $c^{-1}$ on every neurons $f_{k,1}$. Besides, 
 \[\text{Disc}_{\text{safe}}(w_\infty) = \min\Big\{\frac{q -p(c-1)}{c},p\Big\}.\]
\end{proposition}
Note that $w_{\infty}$ is very close to the one of section \ref{specific HAN}: $w_{\infty}^B$ is the same and $w_{\infty}^A$ selects neurons detecting absence of features instead of inhibitory connections. Hence, under the assumptions of Proposition \ref{prop specific han solo}, the conclusions of Theorem \ref{theorem oracle} and Corollary \ref{cor} hold: the network correctly classifies the objects in average and asymptotically.

\section{REGRET} \label{app regret}
\subsection{Details about the regret bounds of EWA and PWA} \label{app regret bounds}
\begin{enumerate}
    \item EWA: the regret bound given in \cite{cesa2006prediction}  holds for losses (\ie negative gains) taking value in $[0,1]$, but a more general demonstration for  only assumed to be bounded is given in \cite{stoltz2010agregation}, and provides the following bound for losses taking value in the interval $[a,b]$ for any $a<b\in \mathbb{R}$:
\[R_M^j \leq \frac{\ln(\abs{I^j})}{\eta^j}+ \eta^j \frac{(b-a)^2}{8}M .\]
With $\eta^j = \frac{1}{b-a}\sqrt{8\ln(\abs{I^j})/M}$, we obtain
\[R_M^j \leq (b-a)\sqrt{\frac{M}{2}\ln(\abs{I^j})}.\]
This choice of $\eta^j$ supposes that we know the time horizon $M$ in advance. If it is not the case, we can use a time-dependent learning rate $\eta_m^j = \frac{1}{b-a}\sqrt{8\ln(\abs{I^j})/m}$, which gives the bound
\[R_M^j \leq \sqrt{2M\ln(\abs{I^j})} + \sqrt{\frac{\ln(\abs{I^j})}{8}}\]
which has the same order of magnitude. Theorem \ref{theorem cvg feasible weight family} can also be adapted with this choice of $\eta^j$. Note that under the assumptions of Theorem \ref{theorem cvg feasible weight family}, for all $j$, $\frac{M}{M^j}=\frac{\abs{\mathcal{O}}}{n^j}$ is bounded by $\abs{\mathcal{O}}$ so the gains take value in $[-\abs{\mathcal{O}},\abs{\mathcal{O}}]$. Hence a good choice of $\eta^j$ is 
\[\eta^j = \frac{1}{2\abs{\mathcal{O}}} \sqrt{8\ln(\abs{I^j})/M}.\]
    \item PWA: the regret bound given in \cite{cesa2006prediction} holds for losses (\ie negative gains) taking value in $[0,1]$. It gives the following inequality:
    \[R_M^j \leq \sqrt{(\beta^j-1)\abs{I^j}^{2/\beta^j}M}. \]
    For losses that are only assumed to be taking value in the interval $[a,b]$ for any $a<b\in \mathbb{R}$, we can translate them thanks to the function $x \mapsto \frac{x-a}{b-a}$. We obtain the following bound:
\[R_M^j \leq (b-a) \sqrt{(\beta^j-1)\abs{I^j}^{2/\beta^j}M} \] 
The bound is optimal for $\beta^j=2\ln(\abs{I^j})$, which gives
    \[R_M^j \leq (b-a)\sqrt{e(2\ln(\abs{I^j})-1)M}. \]
\end{enumerate}
Then both bounds can be written in the form of Assumption \ref{assumption reg} by bounding $\abs{I^j}$ by $\abs{I}$.

\subsection{Interpretation of the regret} \label{app interpretation}

In the HAN Solo framework, the regret can be interpreted in terms of the neurons activity. The \emph{neuronal discrepancy} of neuron $j$ in a network governed by weights $q^j_{1:M}$ is defined by
\begin{equation}\label{neur disc}
    \text{disc}^j_M(q^j_{1:M}):=\widehat{P^{j,j}_M(q^j_{1:M})} - \dfrac{1}{\abs{J}-1}\sum\limits_{j'\neq j} \widehat{P^{j,j'}_M(q^j_{1:M})}.
\end{equation}

 It is the difference between the average estimated spiking probability of neuron $j$ over the objects belonging to class $j$ and the average estimated spiking probability of neuron $j$ over objects belonging to other classes, normalised by the number of objects of each class. It gives information about how much neuron $j$ spikes more than usual when presented objects in category $j$ and when the weights $q^j_{1:M}$ are used. It is a local information (because at neuron $j$). 

Since $ \sum\limits_{m=1}^M \sum\limits_{i\in I^j} q^{i\to j}_m g^{i\to j}_m= M \text{disc}^j_M(q^j_{1:M})$, we have the following {\it interpretation of regret in terms of discrepancy}:
\begin{equation*}
\frac{R^j_M}{M} = \max_{q^j\in \mathcal{X}^j} \text{disc}^j_M(q^j) - \text{disc}^j_M(w^j_{1:M})
\end{equation*}
 where $q^j$ is identified with the constant family of weights $(q^j)_{1\leq m \leq M}$. Therefore, the regret gives information about the proximity of the 
neuronal discrepancy of neuron $j$ under HAN Solo with 
the maximum possible neuronal discrepancy of neuron $j$ with constant weights. Note that this interpretation is made possible thanks to the choice $\varphi = \text{Id}$. This is local information because it is only about the activity of neuron $j$. To understand the global behavior of the network, we need the class and network dicrepancies.

\section{DETAILS ABOUT COMPONENT-CUE} \label{app cc}

The original Component-Cue algorithm \citep{gluck1988conditioning} is an artificial neural network, with three layers: an input layer receiving the stimuli, an intermediate layer decomposing the stimuli into several features, and an output layer made of category nodes. The intermediate and output layers are linked by weights, that the network updates to learn to classify objects.

Let us detail the Component-Cue algorithm. Let $I$ be the set of features (which is also the set of intermediate neurons), $J$ the set of classes (which is also the set of output neurons), $w^{i\to j}_m$ the synaptic weight between neurons $i\in I$ and $j\in J$ when presented with the $m^{th}$ object. Let 
\[a^i_m := \left\{
    \begin{array}{ll}
    1 \text{ if object $o(m)$ has feature $i$} \\
    0 \text{ otherwise.}
    \end{array}
\right. \]
When presented with the $m^{th}$ object, the output neuron $j$ is activated by the quantity 
\[O_m^j := \sum_{i\in I} a^i_m w^{i\to j}_m\]
and object $o(m)$ is classified in class $j$ with probability
\[\frac{e^{\phi O_m^j}}{\sum\limits_{l \in J} e^{\phi O^l_m}}\]
where $\phi$ is a parameter to choose. Then the weights are updated according to the formula
\[w^{i\to j}_{m+1}= w^{i\to j}_m + \lambda_w a^i_m (\tau_m^j - O_m^j)\]
where $\tau_m^j := \left\{
    \begin{array}{ll}
    1 \text{ if $o(m)\in j$ } \\
    -1 \text{ otherwise}
    \end{array}
\right. $ and $\lambda_w$ is the learning rate. It is a gradient descent step.

 Note that the choice of the parameters in Component-Cue is tricky, and that the behavior of the algorithm (learning or not) highly depends on this choice (see details in \cite{mezzadri:tel-03219311}).

 \textbf{Comparison with HAN:} The structure of Component-Cue is very similar to the one of our network; however, Component-Cue does not have a spiking neuronal network interpretation. Indeed, in HAN, spike trains replace the quantities $a^i_m$, which are real numbers. Besides, unlike HAN, Component-Cue has no theoretical guarantee to correctly classify the objects. 

\section{KALIKOW DECOMPOSITION} \label{sec Kalikow}

Let us detail Kalikow decomposition in the case of a general discrete-time linear Hawkes process without inhibition \citep{ost2020sparse}. Let $I$ the set of neurons, $j\in I$ a neuron. Then the spiking probability of neuron $i$ at time $0$ knowing the past is

\[p_i(X) = \nu_i + \sum_{s\in \mathbb{Z}_-^*} h_{j\to i}(-s) X_{j,s}\]
where $X=(X_{j,s})_{j\in I, s\in \mathbb{Z}^*_-}$ is the network past activity, $\nu_i \geq 0$ is the spontaneous activity of neuron $i$, and the functions $h_{j\to i}$ are such that
\begin{itemize}
    \item $h_{j \to i}(s) \geq 0$ for $s\in \mathbb{N}^*$
    \item $\sum_{s \in \mathbb{N}} h_{j\to j}(s) + \nu_i \leq 1$.
\end{itemize}
Then the Kalikow decomposition of $p_i(X)$ is
\[
p_i(X)=\lambda_i(\emptyset)p_i^\emptyset + \sum_{v\in \mathcal{V}, v\neq \emptyset} \lambda_i(v)p_i^v(X)
\]
where $\mathcal{V}= \{\{(j,s)\}, (j,s)\in I\times \mathbb{Z}^*_- \}\cup \{\emptyset\}$ is a family of neighborhoods, $\lambda_i(v)$ is non negative, $\lambda_i(\emptyset) = 1 - \sum_{s \in \mathbb{N}} h_{j\to j}(s)$, $p_i^\emptyset = \frac{\nu_i}{\lambda_i(\emptyset)}$, $\lambda_i(\{(j,s)\}) = h_{j\to i}(-s)$ and $p_i^{\{j,s\}}(X) = X_{j,s}$. 

Then the Kalikow decomposition of $p_i(X)$ is such that $\lambda_i(v) \geq 0$ for all $v\in \mathcal{V}$, $\sum_{v\in \mathcal{V}} \lambda_i(v) = 1$, $p_i^\emptyset \geq 0$ and $p_i^v(X) \geq 0$ for all $v$. Hence, thanks to Kalikow decomposition, the activity of neuron $i$ can be simulated the following way.
\begin{itemize}
    \item A neighborhood $v\in \mathcal{V}$ is drawn thanks to the probability distribution $\lambda$.
    \item Then neuron $i$ spikes with probability $p_i^v(X)$ (resp. $p_i^\emptyset$ if $v=\emptyset)$.
\end{itemize}

In the HAN Solo case, there is no spontaneous activity. The neighborhoods are the tuples $\{(i,s)\}$ where $i\in I^j$ and $s\in \{t-1,\dots, t-K\}$.

\section{OTHER POSSIBLE GAIN} \label{app gain}

In the HAN Solo framework, another choice of gain is possible: 
\begin{align*} 
    &g^{i\to j}_m = \left\{
    \begin{array}{lll}
          \frac{N^{i\to j}_m}{N w^{i\to j}_m} \times \frac{M}{M^j} &\text{if } o(m)\in j  \\ 
          \\
        -\frac{N^{i\to j}_m}{N w^{i\to j}_m} \!\times\! \frac{M}{M^{j'}} \!\times\! \frac{1}{\abs{J}-1} &\text{if } o(m) \in j' \neq j
    \end{array}
\right.
\end{align*}
if $w^{i\to j}_m > 0$ and
\[g^{i\to j}_m =0\]
otherwise, where $N_m^{i\to j}$ is the amount of times neuron $j$ spiked after choosing 
input neuron $i$ in Kalikow decomposition when presented with the $m^{th}$ object.
Note that $\frac{N^{i\to j}_m}{N w^{i\to j}_m}$ is also an estimator of $p_m^i$: indeed, 
knowing the weights $w^j_m$, the variable $N^{i\to j}_m$ follows a binomial distribution 
with parameter $N$ and $p_m^i w_m^{i\to j}$. 
Using this gain, the estimator that we consider for $p_m^i$ is
\[ \widehat{p^i_m} = \frac{N^{i\to j}_m}{N w^{i\to j}_m}.\]
The neuronal discrepancy \eqref{neur disc} becomes
\[\text{disc}^j_M(w^j_{1:M}) = \frac{1}{M^j}\sum_{m, \text{ }o(m)\in j} \frac{N^j_m}{N}
  - \frac{1}{\abs{J}-1} \sum_{j'\neq j} \frac{1}{M^{j'}}\sum_{m, \text{ }o(m)\in j'} \frac{N^{j}_m}{N} \]
where $N^k_m$ is the number of spikes emitted by neuron $k$ during the presentation of the $m^{th}$ object, and the class discrepancy becomes 
 \[\text{Disc}^j_M(w_{1:M}) = \frac{1}{M^j}\sum_{m, \text{ }o(m)\in j} \frac{N^j_m}{N}
  - \frac{1}{\abs{J}-1} \sum_{j'\neq j} \frac{1}{M^j}\sum_{m, \text{ }o(m)\in j} \frac{N^{j'}_m}{N}.\]
Therefore, with this choice of gain, the class discrepancy directly compares
the number of spikes emitted by output neurons, which gives a better indicator about the network's classifications because the rule to classify objects is precisely about the number of spikes, so its interpretation is easier.

However, this gain causes difficulties because of the division by $w^{i\to j}_m$: when $w^{i \to j}_m$ 
is close to zero, the network behaviour is difficult to study theoretically.
First of all, the gains are not bounded anymore so the regret bounds will depend on the 
supremum of the gains, which can be large. Besides, an error term in 
$O\Big(\frac{\ln(M)}{Nw_m^{i\to j}} + \sqrt{\frac{\ln(M)}{Nw_m^{i\to j}}}\Big)$ 
appears in the regret bound instead of the error term $E(N,M,\alpha)$, which was in $O(\frac{1}{NM})$.
This new error is much worse, and converges to zero only if $N>>\frac{\ln(M)}{w_{i\to j}^m}$.
However, the weights of connections corresponding to experts which does not have optimal gains
tend to converge to zero, as stated in Theorem \ref{theorem cvg feasible weight family}.
Hence, we cannot be assured to be in this favorable regime. Besides, this gain cannot be generalized in the HAN framework.

\section{PROOFS} \label{proofs}
\subsection{Proof of Theorem \ref{theorem oracle}}

We need a preliminary proposition.

\begin{proposition}[Regret bound] \label{prop disc bound}
    Suppose Assumptions \ref{assumption reg} and \ref{assump proportions} hold. Then for all $j\in J$,
\[\text{disc}_M^j(w_{1:M}^j) \geq \max_{q^j\in \mathcal{X}^j}\text{disc}_M^j(q^j) -  E_{\text{reg}}(M) \quad \text{a.s.}\]
where $q^j$ is identified with the constant family of weights $(q^j)_{1\leq m \leq M}$ and
\[E_{\text{reg}}(M):= K(\abs{I},(1+(1+\abs{J})^{-1})\xi^{-1}) M^{-1/2}.\]
\end{proposition}

\begin{proof}
    According to Assumption \ref{assump proportions}, for all $j'\in J$, $M^{j'}/M\geq \xi$. Hence the gains $g^{i\to j}_m$ take value in $[-\frac{1}{\xi(\abs{J}-1)},\frac{1}{\xi}]$ a.s. so according to Assumption \ref{assumption reg}, we have
    \[R_M^j \leq K(\abs{I},(1+(\abs{J}-1)^{-1})\xi^{-1}) \sqrt{M} \quad \text{a.s.}\]
    We get the result by dividing the previous inequality by $M$.
\end{proof}

Let $q\in \mathcal{Q}$. We want to bound from below $\text{Disc}_M(w_{1:M})$. We have almost surely
\begin{align*}
\text{Disc}_M(w_{1:M})&=\frac{1}{\abs{J}}\sum_{j\in J} \text{Disc}_M^j(w_{1:M})\\
   & =
     \frac{1}{\abs{J}}\sum_{j\in J} \widehat{P^{j,j}_M(w^j_{1:M})} - \frac{1}{\abs{J}}\sum_{j\in J}\frac{1}{\abs{J}-1}\sum_{j'\neq j} \widehat{P^{j',j}_M(w^{j'}_{1:M})}
\end{align*}

Let us exchange the name of the indexes $j$ and $j'$ in the second term.

\begin{align*}
   \text{Disc}_M(w_{1:M})  =& \frac{1}{\abs{J}}\sum_{j\in J} \widehat{P^{j,j}_M(w^j_{1:M})} - \frac{1}{\abs{J}}\sum_{j'\in J}\frac{1}{\abs{J}-1}\sum_{j\neq j'} \widehat{P^{j,j'}_M(w^j_{1:M})}
\end{align*}

Let us exchange the sums in the second term. 

\begin{align*}
    \text{Disc}_M(w_{1:M})=& \frac{1}{\abs{J}}\sum_{j\in J}\widehat{P^{j,j}_M(w^j_{1:M})} - \frac{1}{\abs{J}}\sum_{j\in J}\frac{1}{\abs{J}-1}\sum_{j'\neq j}  \widehat{P^{j,j'}_M(w^j_{1:M})} \\
    =& \frac{1}{\abs{J}}\sum_{j\in J} \Big( \widehat{P^{j,j}_M(w^j_{1:M})} - \frac{1}{\abs{J}-1}\sum_{j'\neq j}  \widehat{P^{j,j'}_M(w^j_{1:M})} \Big) \\
    =&\frac{1}{\abs{J}}\sum_{j\in J} \text{disc}_M^j (w^j_{1:M})\\
    \geq &\frac{1}{\abs{J}}\sum_{j\in J} \Big(\text{disc}_M^j(q^j)-E_{\text{reg}}(M)\Big) 
\end{align*}
thanks to Proposition \ref{prop disc bound}.

\begin{align*}
     \text{Disc}_M(w_{1:M}) \geq & \frac{1}{\abs{J}}\sum_{j\in J} \Big(\widehat{P^{j,j}_M(q^j)} - \frac{1}{\abs{J}-1}\sum_{j'\neq j} \widehat{P^{j,j'}_M(q^j)} \Big)-E_{\text{reg}}(M) \\
     =&\frac{1}{\abs{J}}\sum_{j\in J} \widehat{P^{j,j}_M(q^j)} - \frac{1}{\abs{J}}\sum_{j\in J}\frac{1}{\abs{J}-1}\sum_{j'\neq j} \widehat{P^{j,j'}_M(q^j)}-E_{\text{reg}}(M)
\end{align*}
Let us exchange the sums in the second term.

\begin{align*}
    \text{Disc}_M(w_{1:M}) \geq\frac{1}{\abs{J}}\sum_{j\in J} \widehat{P^{j,j}_M(q^j)}& - \frac{1}{\abs{J}}\sum_{j'\in J}\frac{1}{\abs{J}-1}\sum_{j\neq j'} \widehat{P^{j,j'}_M(q^j)}-E_{\text{reg}}(M)
\end{align*}

Let us exchange the name of the indexes $j$ and $j'$ in the second term.

\begin{align*}
     \text{Disc}_M(w_{1:M}) \geq&\frac{1}{\abs{J}}\sum_{j\in J} \widehat{P^{j,j}_M(q^j)} - \frac{1}{\abs{J}}\sum_{j\in J}\frac{1}{\abs{J}-1}\sum_{j'\neq j} \widehat{P^{j',j}_M(q^{j'})}-E_{\text{reg}}(M) \\
     =&\frac{1}{\abs{J}}\sum_{j\in J} \Big(\widehat{P^{j,j}_M(q^j)} - \frac{1}{\abs{J}-1}\sum_{j'\neq j} \widehat{P^{j',j}_M(q^{j'})} \Big)-E_{\text{reg}}(M) \\
     =&\text{Disc}_M(q)  -E_{\text{reg}}(M)
\end{align*}

Now we want to compare $\text{Disc}_M(q)$ and $\text{Disc}_{\text{safe}}(q)$. We need the following result.

\begin{proposition} \label{prop hoef}
    Let $\alpha>0$. Suppose Assumption \ref{assump proportions} holds. Then
    \[\mathbb{P}\Big(\forall j\in J,  \forall i\in I, \Big|\frac{1}{M^j}\sum_{m, o(m)\in j} 
    (\widehat{p^i_m} - p^i_m) \Big|\leq \sqrt{\ln\Big(\frac{2\abs{I}\abs{J}}{\alpha}\Big)\frac{1}{2\xi NM}}\Big) \geq 1-\alpha.\]
\end{proposition}

\begin{proof}
    Let $j\in J$. The variables $(X_{m,t}^i)_{1\leq t \leq N, m \text{ s.t. } o(m)\in j}$ are independent bounded by $1$, of mean $p_m^i$. Hence, according to Hoeffding's inequality, for all $\beta>0$
    \[\mathbb{P}\Big(\Big|\frac{1}{NM^j}\sum_{m, o(m)\in j}\sum_{t=1}^N (X_{m,t}^i - p_m^i)\Big| \geq \beta \Big) \leq 2 e^{-2\beta^2NM^j}.\]
    According to Assumption \ref{assump proportions}, $M^j\geq \xi M$ so
    \[\mathbb{P}\Big(\Big|\frac{1}{NM^j}\sum_{m, o(m)\in j}\sum_{t=1}^N (X_{m,t}^i - p_m^i)\Big| \geq \beta \Big) \leq 2 e^{-2\beta^2NM\xi}.\]

    Let $D^c$ be the event $\{\exists j\in J, \exists i\in I, |\frac{1}{NM^j}\sum_{m, o(m)\in j} \sum_{t=1}^N (X_{m,t}^i - p^i_m) |\geq\beta \}.$ Then
    \begin{align*}
        \mathbb{P}(D^c) &\leq \sum_{j\in J, i\in I} \mathbb{P}\Big(\frac{1}{NM^j}\sum_{m, o(m)\in j}\sum_{t=1}^N (X_{m,t}^i - p_m^i) \geq \beta \Big)\\
        &\leq 2\abs{I} \abs{J}e^{-2\beta^2NM\xi}.
    \end{align*}
 Let us choose $\beta$ such that $2\abs{I} \abs{J}e^{-2\beta^2NM\xi}=\alpha$, \ie $\beta= \sqrt{\ln\Big(\frac{2\abs{I}\abs{J}}{\alpha}\Big)\frac{1}{2NM\xi}}$.

    On $D$, for all $j\in J, i\in I$
    \[\Big|\frac{1}{M^j}\sum_{m, o(m)\in j} (\widehat{p^i_m} - p^i_m) \Big|\leq \sqrt{\ln\Big(\frac{2\abs{I}\abs{J}}{\alpha}\Big)\frac{1}{2NM\xi}}\]
    so we can conclude.
\end{proof}

Let $\alpha>0$. Let us work on the event $D$ defined in the proof of Proposition \ref{prop hoef}. Let
\[p_m^j(q^j) := \sum\limits_{i\in I^j} q^{i\to j} p^i_m\]
the spiking probability of neuron $j$ with weights $q^j$ when presented with the $m^{th}$ object,
\[P^{j',j}_M(q^{j'}):= \dfrac{1}{M^{j}}\sum\limits_{m, \text{ }o(m)\in j}
  p^{j'}_m(q^{j'}_m)\]
  the average spiking probability of neuron $j'$ when presented with objects in class $j$. 
 For all $j,j'\in J$, 

  \begin{align*}
    |\widehat{P^{j',j}_M(q^{j'})} - P^{j',j}_M(q^{j'})| &\leq \sum_{i\in I^{j'}} q^{i\to j'} |\frac{1}{M^j} \sum_{m, o(m)\in j} ( \widehat{p^i_m } - p^i_m ) |\\
    &\leq \sqrt{\ln\Big(\frac{2\abs{I}\abs{J}}{\alpha}\Big)\frac{1}{2\xi NM}} 
  \end{align*}
Thanks to Proposition \ref{prop hoef}. Let \[\overline{\text{Disc}}^j_{M}(q^j) = P^{j,j}_M(q^j) - \dfrac{1}{\abs{J}-1}\sum\limits_{j'\neq j} P^{j',j}_M(q^{j'})\]
and
\[\overline{\text{Disc}}_M(q) = \frac{1}{\abs{J}}\sum_{j\in J}\overline{\text{Disc}}_M^j(q^j).\]
For all $j\in J$ we have
\[|\overline{\text{Disc}}^j_{M}(q^j) - \text{Disc}^j_{M}(q^j)| \leq 2 \sqrt{\ln\Big(\frac{2\abs{I}\abs{J}}{\alpha}\Big)\frac{1}{ 2\xi NM}}\]
so
\[|\overline{\text{Disc}}_{M}(q) - \text{Disc}_{M}(q) |\leq 2 \sqrt{\ln\Big(\frac{2\abs{I}\abs{J}}{\alpha}\Big)\frac{1}{2\xi NM}}.\]

Besides, since $q$ is a feasible weight family, 
\[\overline{\text{Disc}}_{M}(q) \geq \text{Disc}_{\text{safe}}(q).\]

Finally, 
\[\text{Disc}_M(w_{1:M}) \geq \text{Disc}_{\text{safe}}(q)  - E_{\text{reg}}(M) - E(N,M,\alpha).\]
This is true for any feasible weight family $q$ and $\mathbb{P}(D)\geq 1-\alpha$ so we can conclude.

\subsection{Proof of Theorem \ref{theorem cvg feasible weight family}}

Let
\begin{itemize}
    \item $ \overline{g}^{i^+\to j}_m := \left\{
        \begin{array}{ll}
           p_i^m \times \frac{M}{M^j} &\text{if } o(m)\in j \\
            - p_i^m \times \frac{1}{\abs{J}-1}\times \frac{M}{M^{j'}} &\text{if } o(m)\in j'\neq j
        \end{array}
    \right.$
    \item $\overline{g}^{i^-\to j}_m = - \overline{g}^{i^+\to j}_m $
    \item $\overline{G}^{i^\bullet \to j}_m=\sum\limits_{m'=1}^m \overline{g}^{i^\bullet \to j}_{m'}$
    \item $\overline{w}^{i^\bullet \to j}_{m+1}=\dfrac{\text{exp}(\eta^j \overline{C}_{i^\bullet \to j}^{m})}{\sum\limits_{k\in I_j} \text{exp}(\eta^j \overline{C}_{k \to j}^{m})}.$
\end{itemize}

We need the following proposition:
\begin{proposition} \label{prop bar}
    Suppose each nature of object is presented the same amount of times. Let $\alpha>0$. Then with EWA with $\eta^j=\frac{1}{\abs{\mathcal{O}}}\sqrt{2\frac{\ln(\abs{I^j})}{M}}$, we get
    \[\mathbb{P}\Big(\forall j\in J, i^\bullet \in I^j, |w^{i^\bullet \to j}_{M+1} - \overline{w}^{i^\bullet \to j}_{M+1} | \leq 
    \abs{I^j}\sqrt{\ln\Big(\frac{2\abs{I}\abs{J}}{\alpha}\Big)\frac{\ln(\abs{I^j})}{\abs{\mathcal{O}}N}}\Big) \geq 1-\alpha\]
\end{proposition}

\begin{proof}
    Let $j\in J$, $i^\bullet \in I^j$, $h^l:\mathbb{R}^{\abs{I^j}}\mapsto \mathbb{R}$ such that $h^l(x_1,\dots,x_{\abs{I^j}}) = \frac{\exp(\eta^j x_l)}{\sum_{k=1}^{\abs{I^j}} \exp(\eta^j x_k)} $. Then for all $(x_1,\dots,x_{\abs{I^j}})$, 
    \[||\nabla h^l(x_1,\dots,x_{\abs{I^j}})||\leq \eta^j \sqrt{\abs{I^j}}.\]

    Besides, according to the mean value theorem, for $i^\bullet \in I^j$
    \begin{align*}
       |w_{i^\bullet \to j}^{M+1} - \overline{w}_{i^\bullet \to j}^{M+1} | &= |h^l((G_{k\to j}^M)_{k\in I^j}) - h^l((\overline{G}_{k\to j}^M)_{k\in I^j})| \\
       &\leq \eta^j \sqrt{\abs{I^j}} ||(G_{k\to j}^M)_{k\in I^j} - (\overline{G}_{k\to j}^M)_{k\in I^j}||.
    \end{align*}

Besides, each nature of object is presented the same amount of times so Assumption \ref{assump proportions} holds with $\xi = \frac{1}{\abs{\mathcal{O}}}$. Hence according to Proposition \ref{prop hoef}, with probability $1-\alpha$ we have that for all $j\in J$, $k \in I^j$,
    \begin{align*}
        |G_{k\to j}^M - \overline{G}_{k\to j}^M| &\leq M \sqrt{\ln\Big(\frac{2\abs{I}\abs{J}}{\alpha}\Big)\frac{1}{2\xi NM}} \\
        &=  \sqrt{\ln\Big(\frac{2\abs{I}\abs{J}}{\alpha}\Big)\frac{M\abs{\mathcal{O}}}{2N}}
    \end{align*}
\ie
  \begin{align*}
       |w^{i^\bullet \to j}_{M+1} - \overline{w}^{i^\bullet \to j}_{M+1} |
       &\leq \eta^j \abs{I^j}  \sqrt{\ln\Big(\frac{2\abs{I}\abs{J}}{\alpha}\Big)\frac{M\abs{\mathcal{O}}}{2N}}.
    \end{align*}
Then we find the result by replacing $\eta^j$ by its value $\frac{1}{\abs{\mathcal{O}}}\sqrt{2\frac{\ln(\abs{I^j})}{M}}$.

\end{proof}

\begin{proposition} \label{prop lim M}
    Suppose each nature of object is presented the same amount of times: for all $o\in \mathcal{O}$, $\abs{\{m, o(m)=o\}}=\frac{M}{\abs{\mathcal{O}}}$. 
Then for all $j\in J$:
\\$\bullet$ if $\tilde{I}^j\neq I^j$, then
$$|\overline{w}^{i^\bullet \to j}_{M+1}-w_\infty^{i^\bullet \to j}|\leq E^j_{\text{EWA}}(M)$$
where 
$$E^j_{\text{EWA}}(M) =\max\Big\{1, \frac{\abs{I^j}}{\abs{\Tilde{I}^j}}-1\Big\}\frac{1}{\abs{\Tilde{I}^j}}e^{-\frac{\gamma^j}{\abs{\mathcal{O}}}\sqrt{2\ln(\abs{I^j})M}}$$  \\
$\bullet$ if $\tilde{I}^j=I^j$, then 
$$\overline{w}^{ i^\bullet \to j}_{M+1} = \abs{I^j}^{-1}.$$ 
\end{proposition}

\begin{proof}
    First, let's prove that for all $m\geq 1$, for all $j\in J$ and $i\in I^j$
\begin{equation} \label{formule credit cumulé}
    \overline{G}^{i^\bullet \to j}_M=d^{i^\bullet \to j}  M.
\end{equation}
Indeed,
\begin{align*}
    \overline{G}^{i^+\to j}_M &= \sum_{m=1}^M \overline{g}^{i^+\to j}_m \\
    &= \sum_{m, \text{ }o(m) \in j} \frac{M}{M^j}\times p^i_m - \sum_{j'\neq j}\sum_{m, \text{ }o(m) \in j'} \frac{1}{\abs{J}-1}\times \frac{M}{M^{j'}}\times p^i_m \\
    &= M\Big( \frac{1}{M^j} \sum_{m, \text{ }o(m) \in j}  p^i_m - \frac{1}{\abs{J}-1} \sum_{j'\neq j} \frac{1}{M^{j'}}\sum_{m, \text{ }o(m) \in j'} p^i_m \Big) \\
    &=M\Big( \frac{1}{M^j} \sum_{o\in j}\sum_{m, \text{ }o(m)=o}  p^i_o - \frac{1}{\abs{J}-1} \sum_{j'\neq j} \frac{1}{M^{j'}}\sum_{o\in j'}\sum_{m, \text{ }o(m)=o} p^i_o \Big) 
\end{align*}
Besides, each kind of object is presented the same amount of times, so for all $j\in J$, 
\[M^j=n^j \times \frac{M}{\abs{O}}\]
so we have
\begin{align*}
    \overline{G}^{i^+\to j}_M &= M\Big( \frac{1}{n_j} \sum_{o\in j}  p_i^o - \frac{1}{\abs{J}-1} \sum_{j'\neq j} \frac{1}{n_{j'}}\sum_{o\in j'}p_i^o\Big) \\
    &= d^{i^+\to j} M.
\end{align*}
 Similarly we have
\[\overline{G}^{i^-\to j}_M =  d^{i^-\to j} M.\]
Let $d_{\text{max}}^j:= \max\limits_{i^\bullet \in I^j} d^{i^\bullet \to j}$.

\textbf{Case $\Tilde{I}^j=I^j$}. Then for all $i^\bullet \in I_j$,
\begin{align*}
    \overline{w}^{i^\bullet \to j}_{M+1} = \dfrac{\exp(\eta^j d_{\text{max}}^j  M)}{\sum\limits_{k \in I^j}\exp(\eta^j d_{\text{max}}^j  M)}
    = \frac{1}{\abs{I^j}} .
\end{align*}

\textbf{Case $\Tilde{I}^j \neq I^j$}. Then
 \begin{align*}
  \overline{w}^{i^\bullet \to j}_{M+1} = \dfrac{e^{\eta^j M d^{i^\bullet \to j} }}{\sum\limits_{k\in I^j}e^{\eta^j M d^{k\to j} }}
  \leq \frac{e^{\eta^j M d^{i^\bullet \to j}  }}{\abs{\Tilde{I}^j} e^{\eta^j Md_{\text{max}}^j}} 
  = \frac{1}{\abs{\Tilde{I}^j}}e^{-\eta^j M (d_{\text{max}}^j-d^{i^\bullet \to j})} .
\end{align*}
Thus
\begin{equation} \label{eq_proof}
    0 \leq  \overline{w}^{i^\bullet \to j}_{M+1} \leq \frac{1}{\abs{\Tilde{I}^j}}e^{-\eta^j M (d_{\text{max}}^j-d^{i^\bullet \to j})}.
\end{equation}
Let $i^\bullet \in \Tilde{I}^j$, $ d_{\text{max bis}}^j:= \max\limits_{k\in I^j\setminus \Tilde{I}^j} d^{k\to j}$. 
\begin{align*}
  \overline{w}^{i^\bullet \to j}_{M+1} &= \dfrac{e^{\eta^j Md^{i^\bullet \to j}}}{\sum\limits_{k\in I_j}e^{\eta^j M d^{k\to j}}} \\
  &\geq  \frac{e^{\eta^j Md_{\text{max}}^j}}{\abs{\Tilde{I}^j}e^{\eta^j Md_{\text{max}}^j } + (\abs{I^j} - \abs{\Tilde{I}^j})e^{\eta^j Md_{\text{max bis}}^j }} \\
  &= \frac{1}{\abs{\Tilde{I}^j}}\frac{1}{1+\frac{\abs{I^j}-\abs{\Tilde{I}^j}}{\abs{\Tilde{I}^j}}e^{-\eta^j \gamma^j M }} \\
  &\geq \frac{1}{\abs{\Tilde{I}^j}}\Big(1-\frac{\abs{I^j}-\abs{\Tilde{I}^j}}{\abs{\Tilde{I}^j}}e^{-\eta^j \gamma^j M}\Big) \\
  &= \frac{1}{\abs{\Tilde{I}^j}} - \frac{\abs{I^j}-\abs{\Tilde{I}^j}}{\abs{\Tilde{I}^j}^2}e^{-\eta^j \gamma^j M } ,
\end{align*}
and thanks to ~\eqref{eq_proof}, \[\overline{w}^{i^\bullet \to j}_{M+1} \leq \dfrac{1}{\abs{\Tilde{I}^j}}.\]
Thus 
\[\frac{1}{\abs{\Tilde{I}^j}}- \frac{\abs{I^j}-\abs{\Tilde{I}^j}}{\abs{\Tilde{I}^j}^2}e^{-\eta^j\gamma^j M } \leq  \overline{w}^{i^\bullet \to j}_{M+1} \leq \frac{1}{\abs{\Tilde{I}^j}}.\]
Let $i\in I^j\setminus \Tilde{I}^j$. Then ~\eqref{eq_proof} tells us that
\[ 0 \leq  \overline{w}^{i^\bullet \to j}_{M+1} \leq \frac{1}{\abs{\Tilde{I}^j}}e^{-\eta^j M \gamma^j }.\]

In particular, with $\eta^j=\dfrac{1}{\abs{\mathcal{O}}}\sqrt{2\dfrac{\ln(\abs{I^j})}{M}}$, for all $i^\bullet \in I_j$
\[
\abs{ \overline{w}^{i^\bullet \to j}_{M+1} - w_\infty^{i^\bullet \to j}} \leq  E^j_{\text{EWA}}(M).
\]

\end{proof}

We get the result by combining Proposition \ref{prop bar} and Proposition \ref{prop lim M}.

\subsubsection{Proof of Corollary \ref{cor}}

We have
\begin{align*}
    p^{j,\text{cond}}_{M+1,t}(w^j_{M+1}) -p^{j',\text{cond}}_{M+1,t}(w^{j'}_{M+1}) =& p^{j,\text{cond}}_{M+1,t}(w^j_{M+1})-p^{j,\text{cond}}_{M+1,t}(w^j_{\infty}) \\
    &+ p^{j,\text{cond}}_{M+1,t}(w^j_{\infty}) - p^{j',\text{cond}}_{M+1,t}(w^{j'}_{\infty}) \\
    &+ p^{j',\text{cond}}_{M+1,t}(w^{j'}_{\infty})  -p^{j',\text{cond}}_{M+1,t}(w^{j'}_{M+1}).
\end{align*}

Let $h\in J$.
\begin{align*}
    |p^{h,\text{cond}}_{M+1,t}(w^h_{M+1})&-p^{h,\text{cond}}_{M+1,t}(w^h_{\infty})| \\
    &= \Big|\varphi\Big(\alpha^h + \sum_{i^+ \in I^{h}_+} w^{i^+\to h}_{M+1} \sum_{k=1}^K g_+(k) X^i_{M+1,t-k} -\sum_{i^- \in I^{h}_-} w^{i^-\to h}_{M+1} \sum_{k=1}^K g_-(k) X^i_{M+1,t-k} \Big) \\
    & \hspace{2cm}- \varphi\Big(\alpha^h + \sum_{i^+ \in I^{h}_+} w^{i^+\to h}_{\infty} \sum_{k=1}^K g_+(k) X^i_{M+1,t-k} -\sum_{i^- \in I^{h}_-} w^{i^-\to h}_{\infty} \sum_{k=1}^K g_-(k) X^i_{M+1,t-k} \Big)\Big| \\
    &\leq  L \Big| \sum_{i^+ \in I^{h}_+} (w^{i^+\to h}_{M+1} - w^{i^+\to h}_{\infty} ) \sum_{k=1}^K g_+(k) X^i_{M+1,t-k}  -\sum_{i^- \in I^{h}_-} (w^{i^-\to h}_{M+1} - w^{i^-\to h}_{\infty}) \sum_{k=1}^K g_-(k) X^i_{M+1,t-k} \Big| \\
    &\leq L   \sum_{i^\bullet \in I^{h}} |w^{i^\bullet \to h}_{M+1} - w^{i^\bullet \to h}_{\infty} |  
\end{align*}
because $\varphi$ is $L$-Lipschitz and the variables $X_{M+1,t-k}^i$ are bounded by $1$.
Hence according to Theorem \ref{theorem cvg feasible weight family}, with probability $1-\alpha$, for all $h\in J$, $t\in \{1,\dots,N\}$ we have
\begin{align*}
     |p^{h,\text{cond}}_{M+1,t}(w^h_{M+1})&-p^{h,\text{cond}}_{M+1,t}(w^h_{\infty})| \leq L\abs{I^h}(E^h(N,\alpha) + E^h_{\text{EWA}}(M)).
\end{align*}
Hence with probability $1-\alpha$, for all $j\in J$, $j'\neq j$, $t\in \{1,\dots,N\}$, $o(M+1)\in j$ we have
\begin{align*}
    p^{j,\text{cond}}_{M+1,t}(w^j_{M+1}) -p^{j',\text{cond}}_{M+1,t}(w^{j'}_{M+1}) \geq  p^{j,\text{cond}}_{M+1,t}(w^j_{\infty}) - p^{j',\text{cond}}_{M+1,t}(w^{j'}_{\infty})-  L\sum_{h\in \{j,j'\}} \abs{I^h}(E^h(N,\alpha) + E^h_{\text{EWA}}(M)).
\end{align*}

\subsection{Proofs of Proposition \ref{prop specific han} and Proposition \ref{prop specific han solo}}\label{proof-thm5}

\subsubsection{Proof of Proposition \ref{prop specific han}} \label{proof cvg han}
Here the assumptions of Theorem \ref{theorem cvg feasible weight family} are verified so Theorem \ref{theorem cvg feasible weight family} applies. Let us study the limit family $(w_{\infty}^A, w_{\infty}^B)$. Let us compute the feature discrepancies. We have $\abs{A}=n^c-1$ and $\abs{B}=1$. Here, $\abs{J}=2$ so for all $k\in \{1,\dots, c\}$, $l\in \{1,\dots, n\}$, $d^{f_{k,l}^+\to A}=-d^{f_{k,l}^-\to A}=-d^{f_{k,l}^+\to B} = d^{f_{k,l}^-\to B}$.

Let $k\in \{1,\dots, c\}$ and $l \in \{2,\dots,n\}$. There are $n^{c-1}-1$ objects in $A$ and $1$ in $B$ with the feature $f_{k,1}$, $n^{c-1}$ in $A$ and $0$ in $B$ with the feature $f_{k,l}$. Hence,
\begingroup
\allowdisplaybreaks 
\begin{align*} 
    d^{f_{k,1}^+\to A} &= \frac{n^{c-1}-1}{n^c-1}p - p \\ 
    d^{f_{k,l}^+\to A} &= \frac{n^{c-1}}{n^c-1}p \\ 
    d^{f_{k,1}^-\to A} &=  p - \frac{n^{c-1}-1}{n^c-1}p\\ 
    d^{f_{k,l}^- \to A} &= - \frac{n^{c-1}}{n^c-1}p\\ 
     d^{f_{k,1}^+ \to B} &=  p - \frac{n^{c-1}-1}{n^c-1}p \\ 
    d^{f_{k,l}^+ \to B} &= - \frac{n^{c-1}}{n^c-1}p \\ 
    d^{f_{k,1}^- \to B} &= \frac{n^{c-1}-1}{n^c-1}p - p\\ 
     d^{f_{k,l}^- \to B} &=  \frac{n^{c-1}}{n^c-1}p .
\end{align*}
\endgroup
It is clear that $d^{f_{k,1}^+\to A}< d^{f_{k,l}^+\to A}$ and $ d^{f_{k,l}^- \to A} <  d^{f_{k,1}^-\to A}$. Besides,
\[d^{f_{k,l}^+\to A} < d^{f_{k,1}^-\to A} \iff n > 2.\]
Hence under the condition $n>2$, $\Tilde{I}^A= \{f_{k,1}^-, 1\leq k\leq c\}$ 
and similarly $\Tilde{I}^B= \{f_{k,1}^+, 1\leq k\leq c\}$. So according to
 Theorem \ref{theorem cvg feasible weight family}, $w_{\infty}^A$ is the family such that
  $w^{f_{k,1}^-\to A} = 1/c$ for $k\in \{1,\dots, c\}$ and $w^{i^\bullet\to A} = 0$ for other connections, 
  whereas $w_{\infty}^B$ is the family such that $w^{f_{k,1}^+\to B} = 1/c$ for $k\in \{1,\dots, c\}$ 
  and $w^{i^\bullet \to B} = 0$ for other connections.

Let us look at the conditions under which this is a feasible weight family. When presented with an object having $h$ features $f_{k_1,1}$,..., $f_{k_h,1}$ in common with object $o_B$,   \[p^{A,\text{cond}}_{m,t}(w^A_{\infty}) = \Big(\alpha^A - \frac{1}{c} \sum_{l=1}^h  X^{f_{k_l,1}}_{m,t-1}\Big)_+\]
so under the condition $\alpha^A < 1/c$,
\begin{align*}
    \mathbb{E}[p^{A,\text{cond}}_{m,t}(w^A_{\infty})] &= \alpha^A \mathbb{P}(X^{f_{k_1,1}}_{m,t} = 0,\dots,X^{f_{k_h,1}}_{m,t}=0) \\
    &= \alpha^A (1-p)^h.
\end{align*}
Besides,
 \[p^{B,\text{cond}}_{m,t}(w^B_{\infty}) = \frac{1}{c} \sum_{l=1}^h X^{f_{k_l,1}}_{m,t-1}\]
so
\begin{align*}
    \mathbb{E}[p^{B,\text{cond}}_{m,t}(w^B_{\infty})] &= \frac{h}{c} p.
\end{align*}

Hence $(w_{\infty}^A, w_{\infty}^B)$ is a feasible weight family if and only if we have the two following conditions:
\begin{align}
    \forall 1\leq h \leq c-1, \quad \alpha^A (1-p)^h &>  \frac{h}{c} p \label{cond1 fwf} \\
    \alpha^A (1-p)^c &  \label{cond2 fwf} < p
\end{align}
where \eqref{cond1 fwf} says that $w_{\infty}$ correctly classifies objects in $A$ and \eqref{cond2 fwf} says that that $w_{\infty}$ correctly classifies the object in $B$.

Besides, \eqref{cond1 fwf} is equivalent to 
\begin{equation*}
    \alpha^A (1-p)^{c-1} >  \frac{c-1}{c} p
\end{equation*}
so $w_{\infty}$ is a feasible weight family if and only if $\alpha^A\in \Big(\frac{(c-1)p}{c(1-p)^{c-1}},\frac{p}{(1-p)^c}\Big)$ and $\alpha^A < 1/c$. If so, the safety discrepancy is 
\[\text{Disc}_{\text{safe}}(w_{\infty})= \min \{ \alpha^A (1-p)^{c-1} -  \frac{c-1}{c} p, p -  \alpha^A (1-p)^c\}.\]
Besides,
\[\frac{(c-1)p}{c(1-p)^{c-1}} < \frac{p}{(1-p)^c} \iff \frac{c-1}{c}(1-p) < 1\]
which is always true so the interval is non-empty, and $\frac{1}{c} > \frac{(c-1)p}{c(1-p)^{c-1}}$ if and only if $(c-1)p < (1-p)^{c-1}$. 

To conclude, if $n>2$ and $ (c-1)p < (1-p)^{c-1}$ then there exists $\alpha^A$ such that $w_{\infty}$ is a feasible weight family.

\subsubsection{Proof of Proposition \ref{prop specific han solo}}
Let us compute the feature discrepancies in this new framework. 

Let $k\in \{1,\dots, c\}$ and 
$l \in \{2,\dots,n\}$. There are $n^c-n^{c-1}$ objects in $A$ and 
  $0$ in $B$ without the feature $f_{k,1}$ and $n^c-n^{c-1}-1$ in $A$ and $1$ in $B$
   without the feature $f_{k,l}$. Hence, similarly as in \ref{proof cvg han}, 
\begingroup
\allowdisplaybreaks 
\begin{align*} 
    d^{f_{k,1}\to A} &= \frac{n^{c-1}-1}{n^c-1}p - p \\ 
    d^{f_{k,l}\to A} &= \frac{n^{c-1}}{n^c-1}p \\ 
    d^{\Tilde{f}_{k,1}\to A} &= \frac{n^c-n^{c-1}}{n^c-1}q\\ 
    d^{\Tilde{f}_{k,l} \to A} &= \frac{n^c-n^{c-1}-1}{n^c-1}q- q\\ 
     d^{f_{k,1} \to B} &= p-  \frac{n^{c-1}-1}{n^c-1}p\\ 
    d^{f_{k,l} \to B} &= -\frac{n^{c-1}}{n^c-1}p \\ 
    d^{\Tilde{f}_{k,1} \to B} &= -\frac{n^c-n^{c-1}}{n^c-1}q\\ 
     d^{\Tilde{f}_{k,l} \to B} &= q-\frac{n^c-n^{c-1}-1}{n^c-1}q .
\end{align*}
\endgroup
It is clear that $d^{f_{k,1}\to A} < d^{f_{k,l}\to A}$ and 
$ d^{\Tilde{f}_{k,l} \to A} <  d^{\Tilde{f}_{k,1}\to A} $. Besides, 
\[d^{f_{k,l}\to A} <  d^{\Tilde{f}_{k,1} \to A} \Longleftrightarrow \frac{p}{n-1} < q\]
and 
\[ d^{\Tilde{f}_{k,l}\to B} < d^{f_{k,1} \to B} \Longleftrightarrow q < (n-1)p. \]

Thus, under the hypothesis \[\frac{p}{n-1} < q <  (n-1)p,\]
the neurons having maximal feature discrepancies are neurons $\Tilde{f}_{k,1}$ for $A$, 
and $f_{k,1}$ for $B$. So according to Theorem \ref{theorem cvg feasible weight family}, $w_{\infty}^A$ is the family such that
  $w^{\Tilde{f}_{k,1}\to A} = 1/c$ for $k\in \{1,\dots, c\}$ and $w^{i\to A} = 0$ for other connections, 
  whereas $w_{\infty}^B$ is the family such that $w^{f_{k,1}\to B} = 1/c$ for $k\in \{1,\dots, c\}$ 
  and $w^{i\to B} = 0$ for other connections. With constant weights $(w^A_{\infty},w^B_{\infty})$,
   neurons $A$ and $B$ have the following spiking probabilities.

\begin{center}
    \begin{tabular}{|c|c|c|} \hline
     & $o\in A$ with a $l$ common features with $o_B$  & $o_B$ \\ \hline
    $A$ & $q (c-l)/c$ & $0$ \\ \hline
    $B$ & $p l/c$ & $p$ \\ \hline
    \end{tabular}
\end{center}

Hence with $q > (c-1) p$, the pair $w_{\infty}=(w^A_{\infty}, w^B_{\infty})$ is indeed a feasible weight family.

 Besides, 
\begin{align*}
   \text{Disc}_{\text{safe}}(w_{\infty})&= \min\Big\{\min_{l \in \{0,c-1\}} \Big\{q\frac{c-l}{c} - p\frac{l}{c}\Big\} , p \Big\}
\end{align*}
The second minimum is achieved for $l=c-1$, so
\[\text{Disc}_{\text{safe}}(w_{\infty}) = \min\Big\{\frac{q - p(c-1)}{c}, p \Big\}.\]

\vfill

\end{document}